\documentclass[11pt]{article}

\usepackage[T1]{fontenc}
\usepackage{lmodern}
\usepackage{amsmath,amssymb,amsthm,mathtools}
\usepackage{enumitem}
\usepackage{xcolor}
\usepackage{tikz-cd}
\usepackage{hyperref}

\usepackage{a4wide}

\newcommand{\Q}{\mathbb{Q}}
\newcommand{\Z}{\mathbb{Z}}
\newcommand{\N}{\mathbb{N}}
\newcommand{\R}{\mathbb{R}}
\newcommand{\C}{\mathbb{C}}
\newcommand{\F}{\mathbb{F}}
\newcommand{\A}{\mathbb{A}}
\newcommand{\G}{\mathbb{G}}
\newcommand{\bbP}{\mathbb{P}}
\newcommand{\bbS}{\mathbb{S}}
\newcommand{\OK}{\mathcal{O}_K}
\newcommand{\Oring}{\mathcal{O}}
\newcommand{\Gal}{\operatorname{Gal}}
\newcommand{\GL}{\operatorname{GL}}
\newcommand{\SL}{\operatorname{SL}}
\newcommand{\Aut}{\operatorname{Aut}}
\newcommand{\End}{\operatorname{End}}
\newcommand{\Hom}{\operatorname{Hom}}
\renewcommand{\Im}{\operatorname{Im}}
\newcommand{\card}[1]{\# #1}

\providecommand{\nomenclature}[2]{}
\providecommand{\index}[1]{}

\newtheorem{theorem}{Theorem}[section]
\newtheorem{proposition}[theorem]{Proposition}
\newtheorem{lemma}[theorem]{Lemma}
\newtheorem{corollary}[theorem]{Corollary}
\theoremstyle{definition}
\newtheorem{definition}[theorem]{Definition}
\newtheorem{example}[theorem]{Example}
\newtheorem{notation}[theorem]{Notation}
\newtheorem{conjecture}[theorem]{Conjecture}
\theoremstyle{remark}
\newtheorem{remark}[theorem]{Remark}

\newcommand{\GSp}{\operatorname{GSp}}
\newcommand{\Spec}{\operatorname{Spec}}
\newcommand{\Pic}{\operatorname{Pic}}
\newcommand{\Div}{\operatorname{Div}}
\newcommand{\Princ}{\operatorname{Princ}}
\newcommand{\abGal}[1]{\Gal\big(\overline{#1}/#1 \big)}
\newcommand{\Jac}{\operatorname{Jac}}

\title{An introduction to algebraic groups}

\author{Davide Lombardo}
\date{}

\begin{document}
\maketitle

\begin{abstract}
We review the theory of abelian varieties and, more generally, algebraic groups, with a view toward arithmetic applications. We discuss their models over the rings of integers of number fields, their reductions modulo primes, their torsion points (together with their natural Galois action), and the interplay of these notions through characteristic polynomials of Frobenius.
\end{abstract}

\section{Basic notation and conventions}
We reserve the letter $K$ to denote fields. When $K$ is a number field (that is, a finite extension of the field $\Q$ of rational numbers) or a local field, we denote by $\OK$ its ring of integers. The letters $p$ and $\ell$ will denote prime numbers, while the symbol $v$ will denote a non-zero prime ideal in the ring of integers of some number field (sometimes called a \emph{finite place}). When $K$ is a number field and $v$ is a finite place of $K$, the completion of $K$ at $v$ is denoted by $K_v$. The completion of $\OK$ at $v$ is the ring of integers of $K_v$; it will be denoted by $\Oring_{K_v}$, with residue field $\F_v$.
The symbol $C$ will usually denote a curve with associated Jacobian variety $J$ (to be defined later). 
A \emph{variety} over a field is a scheme of finite type over that field, separated and geometrically integral (i.e., reduced and irreducible after base-change to an algebraic closure). We assume in particular that varieties are geometrically connected. A \emph{nice curve} is a smooth projective variety of dimension 1 (in particular, geometrically connected). Given a scheme $X$ over a field $K$ and a field extension $K \subseteq L$, we denote by $X_L$ the base-change $X \times_{\operatorname{Spec} K} \operatorname{Spec} L$.

\section{Group schemes}\label{sect: group schemes}
The language of group schemes is essential in developing some of the more advanced parts of the arithmetic theory of abelian varieties and, more generally, algebraic groups. Indeed, schemes enter naturally into the picture when we try to make sense of operations like reducing points modulo primes: while for elliptic curves it is possible to argue entirely with equations, without ever mentioning schemes, it seems that when dealing with higher-dimensional abelian varieties the schematic language becomes indispensable. 
For this reason, we will spend some time introducing the general notion of group schemes, starting with the following definition.

\begin{definition}
Let $S$ be a scheme. A \emph{group scheme} over $S$ is an $S$-scheme $\pi \colon \mathcal{G} \to S$ together with three morphisms
\[
m\colon \mathcal{G} \times_{S} \mathcal{G} \to \mathcal{G}, \quad i \colon \mathcal{G} \to \mathcal{G}, \quad e\colon S \to \mathcal{G},
\]
called respectively the \emph{multiplication}, \emph{inverse}, and \emph{unit} maps. These maps fit into commutative diagrams that ensure that, for any $S$-scheme $T$, the maps 
\[
m_T \colon \mathcal{G}(T) \times \mathcal{G}(T) \to \mathcal{G}(T) \quad \text{and}\quad i_T \colon \mathcal{G}(T) \to \mathcal{G}(T)
\]
endow $\mathcal{G}(T)$ with the structure of a group with identity element given by $(T \to S \xrightarrow{e} \mathcal{G}) \in \mathcal{G}(T)$.
Explicitly, associativity of $m$ is expressed by the commutativity of the diagram
\begin{center}
\begin{tikzcd}[row sep=large, column sep=huge]
\mathcal{G} \times_S \mathcal{G} \times_S \mathcal{G} \arrow[r, "{(m,\operatorname{id})}"] \arrow[d, "{(\operatorname{id},m)}"'] 
& \mathcal{G} \times_S \mathcal{G} \arrow[d, "m"] \\
\mathcal{G} \times_S \mathcal{G} \arrow[r, "m"'] 
& \mathcal{G}
\end{tikzcd}
\end{center}
The following two diagrams express that $e$ is a left and right unit (the isomorphisms $S \times_S \mathcal{G} \to \mathcal{G}$ and $\mathcal{G} \times_S S \to \mathcal{G}$ are the canonical ones):
\begin{center}
    \begin{tikzcd}[column sep=large]
S \times_S \mathcal{G} \arrow[r, "{(e,\operatorname{id})}"] \arrow[dr, "\cong"'] 
& \mathcal{G}\times_S\mathcal{G} \arrow[d, "m"] \\
& \mathcal{G}
\end{tikzcd}\hspace{2cm}\begin{tikzcd}[column sep=large]
\mathcal{G} \times_S S \arrow[r, "{(\operatorname{id},e)}"] \arrow[dr, "\cong"'] 
& \mathcal{G}\times_S\mathcal{G} \arrow[d, "m"] \\
& \mathcal{G}
\end{tikzcd}
\end{center}
Finally, the fact that $i$ gives left and right inverses is encoded by
\begin{center}
\begin{tikzcd}[column sep=large]
\mathcal{G} \arrow[r, "{(i,\operatorname{id})}"] \arrow[dr, "e\circ\pi"'] 
& \mathcal{G}\times_S\mathcal{G} \arrow[d, "m"] \\
& \mathcal{G}
\end{tikzcd}
\hspace{2cm}
\begin{tikzcd}[column sep=large]
\mathcal{G} \arrow[r, "{(\operatorname{id}, i)}"] \arrow[dr, "e\circ\pi"'] 
& \mathcal{G}\times_S\mathcal{G} \arrow[d, "m"] \\
& \mathcal{G}
\end{tikzcd}
\end{center}

\end{definition}

There is also an obvious notion of homomorphism of group schemes: if $\mathcal{G}_1, \mathcal{G}_2$ are two group schemes over $S$, with group structure given by the maps $(m_1, i_1, e_1)$ and $(m_2, i_2, e_2)$ respectively, an $S$-scheme morphism $f \colon \mathcal{G}_1 \to \mathcal{G}_2$ is called a \emph{homomorphism} if it satisfies
\[
f \circ e_1 = e_2, \quad f \circ m_1=m_2 \circ (f \times_S f), \quad f\circ i_1 = i_2 \circ f.
\]
These conditions imply that, for any $S$-scheme $T$, the induced map $f \colon \mathcal{G}_1(T) \to \mathcal{G}_2(T)$ is a group homomorphism.

Finally, if $f \colon  \mathcal{G}_1 \to \mathcal{G}_2$ is a homomorphism of group schemes over $S$, then one defines the kernel $\ker f$ in the obvious way, namely as the fibre product in the diagram below:
\[
\begin{tikzcd}
\ker f \arrow[r] \arrow[d] & \mathcal{G}_1 \arrow[d, "f"'] \\
S \arrow[r, "e_2"'] & \mathcal{G}_2
\end{tikzcd}
\]
The $S$-scheme $\ker f$ is then a subgroup scheme of $\mathcal{G}_1$ (with the obvious definition: a subscheme which inherits a structure of $S$-group scheme when endowed with the suitable restrictions of the maps $m_1, i_1, e_1$). 

\medskip

We now make several remarks on these definitions.

\begin{remark}
More informally, when $S$ is the spectrum of a field $K$, a group scheme over $S$ (also called a $K$-group scheme) is a $K$-scheme $G$ whose $\overline{K}$-points form a group, and such that:
\begin{enumerate}
\item the identity element of $G(\overline{K})$ is $K$-rational;
\item the functions $m\colon G(\overline{K}) \times G(\overline{K}) \to G(\overline{K})$ and $i \colon G(\overline{K}) \to G(\overline{K})$ that give the multiplication and inverse in the group are induced by algebraic morphisms $G \times G \to G$ and $G \to G$ defined over $K$.
\end{enumerate}

Group schemes of finite type over a field are often simply called \emph{algebraic groups}.
\end{remark}

\begin{remark}
Roughly speaking, group schemes over a general base $S$ can be thought of as families of algebraic groups (in the sense of the previous remark) parametrised by the points of $S$.
\end{remark}

The following definitions will also be useful.
\begin{definition}\phantom{}
\begin{itemize}
    \item A group scheme $\mathcal{G}$ over the base scheme $S$ is called \emph{finite} (resp.~\emph{finite étale}) if the structure morphism $\mathcal{G} \to S$ is finite (resp.~finite étale).
    \item Given two $S$-group schemes $\mathcal{G}_1, \mathcal{G}_2$, we sometimes denote by $\operatorname{Hom}_S(\mathcal{G}_1, \mathcal{G}_2)$ the set of homomorphisms of $S$-group schemes $\mathcal{G}_1 \to \mathcal{G}_2$. In particular, if $S=\operatorname{Spec} K$, we write $\operatorname{Hom}_K(\mathcal{G}_1, \mathcal{G}_2)$ for $\operatorname{Hom}_S(\mathcal{G}_1, \mathcal{G}_2)$.
    \item Let $K \subseteq L$ be a field extension and $\mathcal{G}$ be a $K$-group scheme, with group structure given by the morphisms $(m, i, e)$. The $L$-scheme $\mathcal{G}_L  \coloneq  \mathcal{G} \times_{\operatorname{Spec} K} \operatorname{Spec} L$ has a natural structure of $L$-group scheme, given by the three morphisms $m_L, i_L, e_L$ obtained as the base-change to $L$ of $m, i, e$ respectively. For simplicity of notation, when $\mathcal{G}_1, \mathcal{G}_2$ are two $K$-schemes and $K \subseteq L$ is a field extension, we shall sometimes write $\operatorname{Hom}_L(\mathcal{G}_1, \mathcal{G}_2)$ for $\operatorname{Hom}_L( (\mathcal{G}_1)_L, (\mathcal{G}_2)_L)$.
\end{itemize}
\end{definition}

\begin{remark}\label{rmk: group schemes are functors}
    One could take a more categorical point of view and define group schemes as group objects in the category of schemes. We will not adopt this approach. Nonetheless, it is useful to observe that a group scheme $\mathcal{G}$ over a scheme $S$ gives a contravariant functor from the category of $S$-schemes to the category of groups. In particular, whenever we have a map of $S$-schemes $f \colon T_1 \to T_2$, we have an induced group homomorphism $f \colon \mathcal{G}(T_2) \to \mathcal{G}(T_1)$, given by
    \[
    \begin{array}{cccc}
    f \colon &\mathcal{G}(T_2) & \to & \mathcal{G}(T_1) \\
    & (T_2 \xrightarrow{x} \mathcal{G}) &  \mapsto & (T_1 \xrightarrow{f} T_2 \xrightarrow{x} \mathcal{G}).
    \end{array}
    \]

\end{remark}

\begin{remark}\label{rmk: etale group schemes}
Let $K$ be a field of characteristic zero and $S=\Spec K$. Finite group schemes over $S$ can be described in very concrete terms as follows. Given a finite group scheme $\mathcal{G}$ we have a finite group $G \coloneq  \mathcal{G}(\overline{K})$ on which $\Gal(\overline{K}/K)$ acts naturally and continuously. Conversely, given a finite group $G$ with a continuous action of $\Gal(\overline{K}/K)$, there is a unique group scheme $\mathcal{G}$ such that $\mathcal{G}(\overline{K})\cong G$ as groups endowed with a continuous action of $\Gal(\overline{K}/K)$. One can show that these two constructions are inverse to one another, so that finite group schemes over $K$ form a category equivalent to the category of finite groups with a continuous action of $\Gal(\overline{K}/K)$. The same description extends to finite étale group schemes over an arbitrary field $K$, replacing $\overline{K}$ with the separable closure $K^{\operatorname{sep}}$. We note here that a theorem of Cartier \cite[\href{https://stacks.math.columbia.edu/tag/047N}{Tag 047N}]{stacks-project} implies that a finite group scheme over $\operatorname{Spec} K$ is always étale if $K$ is of characteristic $0$. More generally, étale group schemes over $K$ correspond to possibly infinite groups $G$, endowed with the discrete topology and a continuous action of $\Gal(K^{\operatorname{sep}}/K)$. The associated group scheme is of finite type if and only if $G$ is finite. For example, the constant group scheme $\coprod_{m\in\Z}\Spec K$, with group law induced by addition in $\Z$, is étale but not of finite type.
\end{remark}

\begin{example} Here are two fundamental examples. Before describing them, we recall that for a scheme $X$, the ring $H^0(X,\mathcal O_X)$ of global sections is, by definition, the ring of global regular functions on $X$. If $X$ is affine, say $X=\operatorname{Spec} R$, then $H^0(X,\mathcal O_X)=R$. In general, one covers $X$ by affine open subschemes $\{\operatorname{Spec} R_i\}_{i\in I}$, and the elements of $H^0(X,\mathcal O_X)$ are precisely collections of elements of the rings $R_i$ that agree on overlaps, i.e., satisfy the appropriate gluing conditions.
\begin{enumerate}
\item The \emph{multiplicative group} $\G_{m,S}$. We start by defining $\G_{m,\Z}$: as a scheme, it is given by $\G_{m,\Z}=\Spec \Z[t, t^{-1}]$. In order to describe its group structure, we need to specify three morphisms $e, i$ and $m$ as above. The identity section $e$ is the map of affine schemes induced by the map of rings 
\[
\begin{array}{ccc}
\Z[t, t^{-1}] & \to & \Z \\
t & \mapsto & 1.
\end{array}
\]
The inverse $i$ is induced by the map of rings
\[
\begin{array}{ccc}
\Z[t, t^{-1}] & \to & \Z[t, t^{-1}] \\
t & \mapsto & t^{-1}
\end{array}
\]
and the multiplication $m$ is induced by
\[
\begin{array}{ccc}
\Z[t, t^{-1}] & \to & \Z[t_1,t_2,t_1^{-1},t_2^{-1}] \\
t & \mapsto & t_1t_2.
\end{array}
\]
Notice that $\Z[t_1,t_2,t_1^{-1},t_2^{-1}] \cong \Z[t_1, t_1^{-1}] \otimes_{\Z} \Z[t_2,t_2^{-1}]$, and the latter is the coordinate ring of $\G_{m,\Z} \times_{\Spec \Z} \G_{m,\Z}$.
One checks that, for any ($\Z$-)scheme $X$, we have
\[
\begin{aligned}
\G_{m,\Z}(X) & = \Hom_{\operatorname{Sch}}(X, \G_{m,\Z}) \\
& = \Hom_{\operatorname{Ring}}(\Z[t,t^{-1}], H^0(X,\Oring_X))=H^0(X,\Oring_X)^\times.
\end{aligned}
\]
Here the second equality follows from the fact that $\G_{m,\Z}$ is affine, while the third is a consequence of the fact that a ring homomorphism $\varphi \colon \Z[t] \to H^0(X,\Oring_X)$ is uniquely determined by $a=\varphi(t)$, and it extends to $\Z[t,t^{-1}]$ if and only if $a$ is invertible in $H^0(X,\Oring_X)$.

The previous equalities justify the name \emph{multiplicative group}: when evaluated on $X=\Spec(A)$, where $A$ is a ring, we get $\G_{m,\Z}(X)=A^\times$. One may also check that the induced maps $m \colon \G_{m,\Z}(X) \times \G_{m,\Z}(X) \to \G_{m,\Z}(X)$ and $i: \G_{m,\Z}(X) \to \G_{m,\Z}(X)$ are the obvious ones. We make this explicit for the multiplication $m$: given two elements $\varphi_1, \varphi_2 \in \G_{m,\Z}(X)$, we would like to compute their product. Considering the diagram
\begin{center}
\begin{tikzcd}[row sep=large, column sep=huge]
X \arrow[bend right=20]{ddr}[swap]{\varphi_1} \arrow[bend left=20]{drr}{\varphi_2} \arrow{dr}[dashed]{(\varphi_1,\varphi_2)} & & \\
& \G_{m,\Z} \times_{\Spec\Z} \G_{m,\Z} \arrow{r} \arrow{d} \arrow{dr}[swap]{m} 
& \G_{m,\Z} \\[1em]
& \G_{m,\Z} 
& \G_{m,\Z}
\end{tikzcd}
\end{center}
we are therefore interested in the composition $m \circ (\varphi_1, \varphi_2)$. We may assume that $X=\Spec(A)$ is affine, in which case the maps $\varphi_1, \varphi_2$ are determined by $a_1 \coloneq \varphi_1(t) \in A^\times$ and $a_2 \coloneq \varphi_2(t) \in A^\times$. We would like to describe $m \circ (\varphi_1, \varphi_2)$ and check that it corresponds to $a_1a_2 \in A^\times$.
Indeed, we have a corresponding diagram of rings
\begin{center}
    \begin{tikzcd}[row sep=large, column sep=huge]
A & \\[1em]
& \Z[t_1,t_1^{-1}] \otimes_{\Z} \Z[t_2,t_2^{-1}] \arrow[ul, dashed] & \Z[t_2,t_2^{-1}] \arrow[l] \arrow[ull, bend right=20, swap, "t_2 \mapsto a_2"] \\[1em]
& \Z[t_1,t_1^{-1}] \arrow[u] \arrow[uul, bend left=20, "t_1 \mapsto a_1"] & \Z[t,t^{-1}] \arrow[ul, "m"]
\end{tikzcd}
\end{center}
so that the composite map $\Z[t,t^{-1}] \xrightarrow{m} \Z[t_1,t_1^{-1}] \otimes_{\Z} \Z[t_2,t_2^{-1}] \to A$ sends $t \mapsto t_1t_2 \mapsto a_1a_2$ as desired.
The verification for the inverse is similar, and we leave it to the reader. Finally, the element of $\G_{m, \Z}(X)$ induced by $e$ is $X \to \Spec \Z \to \Spec \Z[t, t^{-1}]$, which corresponds to the ring map $\Z[t, t^{-1}] \to A$ that sends $t$ to $1 \in A^\times$.
The upshot of this discussion is that, for every scheme $X$, we may endow the set $\G_{m,\Z}(X)=H^0(X,\Oring_X)^\times$ with the group structure induced by $e, i$ and $m$, and this group structure agrees with the natural multiplicative structure of $H^0(X,\Oring_X)^\times$.

Finally, if $S$ is a general base scheme, the multiplicative group over $S$ is simply $\G_{m,S} = \G_{m,\Z} \times_{\Spec \Z} S$, and for any $S$-scheme $X$ we again have a canonical group isomorphism $\G_{m,S}(X) \cong H^0(X,\Oring_X)^\times$.
\item In a similar manner, we may also define the \emph{additive group $\G_{a,S}$}: again it suffices to define $\G_{a,\Z}$, which as a scheme is simply $\Spec \Z[t]$. A calculation similar to the above shows that for any test scheme $X$ we have $\G_{a,\Z}(X)=H^0(X,\Oring_X)$ (as sets). We may then equip $\G_{a,\Z}$ with the structure of a group scheme by endowing it with the morphisms corresponding to the ring maps
\[
\begin{array}{cccc}
e \colon & \Z[t] & \to & \Z, \\
& t & \mapsto & 0
\end{array} \quad\quad 
\begin{array}{cccc}
i \colon & \Z[t] & \to & \Z[t], \\
& t & \mapsto & -t
\end{array} \quad\quad
\begin{array}{cccc}
m \colon & \Z[t] & \to & \Z[t_1,t_2]. \\
& t & \mapsto & t_1+t_2
\end{array}
\]
Finally, one checks that, given elements $\varphi_1, \varphi_2$ in $\G_{a,\Z}(X)$, corresponding to $a_1, a_2 \in H^0(X,\Oring_X)$, the element $m (\varphi_1,\varphi_2)$ of $\G_{a,\Z}(X)$ corresponds to $a_1+a_2 \in H^0(X, \Oring_X)$, and $i(\varphi_1)$ is the element corresponding to $-a_1$. In other words, for every scheme $X$ we have a canonical group isomorphism $\G_{a,\Z}(X) \cong \left(H^0(X,\Oring_X), + \right)$.
\end{enumerate}
\end{example}

With these preliminaries in place, we move on to the kind of group schemes we will mostly be interested in, namely, abelian varieties.


\section{Abelian varieties: geometry}
Let $K$ be a field. An \emph{abelian variety $A$ over $K$} is a reduced, connected, and projective algebraic group. More explicitly, this means that $A$ is a (projective, connected) $K$-variety endowed with a distinguished rational point $0_A \in A(K)$ and algebraic morphisms $m \colon A \times A \to A$, $i \colon A \to A$ that make $A$ into a group scheme.
Before proceeding, we make some comments on this seemingly innocuous definition, which in fact hides quite a bit of sophisticated mathematics.
\begin{remark}\phantom{~}
\begin{enumerate}
\item In the definition, one may replace \emph{projctive} with \emph{proper}. It is a theorem of Weil that the two definitions are equivalent: a reduced, connected, proper algebraic group is automatically projective (for a proof, see \cite[§7]{MilneAV}). The groups $\G_a$ and $\G_m$ are not abelian varieties: they are smooth and geometrically connected, but neither is proper over $K$.
\item Since abelian varieties are by definition connected and possess a $K$-rational point, they are also geometrically connected \cite[\href{http://stacks.math.columbia.edu/tag/04KV}{Tag 04KV}]{stacks-project}.
\item A theorem of Cartier shows that if $\operatorname{char}(K)=0$ all group schemes over $K$ are automatically reduced \cite[Theorem 3.20]{EvGM}; this is not true in general over fields of positive characteristic. For example, if $\operatorname{char}(K)=p>0$, the map $\G_a\to\G_a$ given by $x\mapsto x^p$ is a group homomorphism; its kernel $\alpha_p=\Spec(K[t]/(t^p))$ is not reduced.
\item All \emph{geometrically reduced} group schemes that are locally of finite type over a field are smooth \cite[Proposition 3.17]{EvGM}. Combined with (1) and (3), this shows that if $\operatorname{char}(K)=0$ one may simply define abelian varieties as connected, proper group schemes of finite type over $K$.
\end{enumerate}
\end{remark}

\begin{lemma}[Rigidity lemma]\label{lemma:Rigidity}
Let $f \colon A \times B \to C$ be a morphism of varieties
over a field $K$. If $A$ is proper and
$f\left( A \times \{b_0\} \right) = f\left( \{a_0\} \times B \right) = \{ c_0 \}$
for certain points $a_0 \in A(K), b_0 \in B(K), c_0 \in C(K)$, then $f(A \times B)=\{c_0\}$.
\end{lemma}
\begin{proof}
Choose an open affine neighbourhood $C_0$ of $c_0$. By properness, $\pi \colon A \times B \to B$ is a closed map, hence $Z=\pi(f^{-1}(C \setminus C_0))$ is closed in $B$. A closed point $b$ of $B$ lies outside $Z$ if and only if $f(A \times \{b\}) \subseteq C_0$; by assumption $b_0$ lies in $B \setminus Z$, which is therefore open and nonempty, hence dense (recall that our varieties are geometrically irreducible by definition). Now pick any point $b \in B \setminus Z$ and consider $f(A \times \{b\})$: on the one hand, the image of this map is contained in $C_0$ by what we just said, and on the other, 
since $A \times \{b\}$ is proper and $C_0$ is affine, $f(A \times \{b\})$ is a point. Moreover, by assumption we have
\[
f(A \times \{b\}) \ni f(\{a_0\} \times \{b\} ) \in f(\{a_0\} \times B) = \{c_0\},
\]
hence $f(A \times \{b\})=\{c_0\}$ for all $b$ in the dense set $B \setminus Z$. In particular, $f$ is constantly equal to $c_0$ on the dense open set $A \times (B \setminus Z)$, hence it is constant as claimed.
\end{proof}
As the name suggests, abelian varieties are commutative:
\begin{proposition}\label{prop: AV are commutative}
Any abelian variety is commutative, that is, the two maps
\[
\begin{array}{ccc}
A \times A & \to & A \\
(x,y)&\mapsto & x \cdot y
\end{array}
\]
and
\[
\begin{array}{ccc}
A \times A & \to & A \\
(x,y)&\mapsto & y \cdot x
\end{array}
\]
coincide, where we (temporarily) denote by $\cdot$ the multiplication map on $A$.
\end{proposition}
\begin{proof}
Since we are considering schemes of finite type over a field, it suffices to work at the level of $\overline{K}$-points (notice that $A$ is separated, so two morphisms are equal if and only if they are equal at all closed points). Denote by $x \mapsto x^{-1}$ the inverse map of $A$.
It is enough to show that the image of the map
\[
\begin{array}{ccc}
A(\overline{K}) \times A(\overline{K}) & \to & A(\overline{K}) \\
(x,y)&\mapsto & y \cdot x \cdot y^{-1} \cdot x^{-1}
\end{array}
\]
is the identity element $e_A$ of $A(\overline{K})$. Now notice that the restriction of this map to $\{e_A\} \times A(\overline{K})$ and to $A(\overline{K}) \times \{e_A\}$ is constantly equal to $e_A$, and apply the rigidity lemma (Lemma \ref{lemma:Rigidity}).
\end{proof}

Because of the previous proposition, we will usually denote the group operation on an abelian variety additively, writing $0_A$ (or simply $0$) for the neutral element, $+$ for the addition morphism, and $-x$ for the opposite of $x$ with respect to the group law. We will also denote by $[-1] \colon A \to A$ the inverse map.
\begin{definition}[Translation maps]\label{def:TranslationMaps}For $x \in A(K)$, we denote by $\tau_x \colon A \to A$ the translation map $a \mapsto a+x$.\end{definition}

\subsection{The theorems of the square and of the cube}

We collect in this section some technical results that will be useful in what follows. The reader need not spend much time meditating on these theorems, which are only included for completeness (proofs of all these statements can be found in \cite{Mumford70abelianvarieties}).

\begin{theorem}[Theorem of the cube]\label{thm:Cube}
Let $U, V , W$ be complete geometrically irreducible
varieties over $K$, and let $u_0 \in U(K)$, $v_0 \in V(K)$, $w_0 \in W(K)$ be base points. An invertible sheaf $\mathcal{L}$ on $U \times V \times W$ is trivial if and only if its restrictions to
\[
U  \times V \times \{w_0\}, \quad U \times \{v_0\} \times  W ,\quad \{u_0\} \times  V \times  W
\]
are all trivial.
\end{theorem}

\begin{theorem}[Theorem of the square]\label{thm:Square}
Let $A$ be an abelian variety over a field $K$. For all line bundles $\mathcal{L}$ on $A$ and for all points $a, b \in A(K)$ we have
\[
\tau_{a+b}^* \mathcal{L} \otimes \mathcal{L} \cong \tau_a^* \mathcal{L} \otimes \tau_b^* \mathcal{L}.
\]

\end{theorem}

For the next result we need the following definition.

\begin{definition}\label{def:PicardGroup}
Let $X$ be a $K$-variety.
\begin{enumerate}
\item Let $T$ be a $\overline{K}$-variety. A line bundle $\mathcal{M}$ on $X_{\overline{K}}\times T$ gives a family of line bundles on $X_{\overline{K}}$: for each $t\in T(\overline{K})$, we obtain the line bundle $\mathcal{M}_t$ by restricting $\mathcal{M}$ to $X_{\overline{K}}\times\{t\}$. Two line bundles on $X$ are \emph{algebraically equivalent} if, after base change to $\overline{K}$, they occur as two members of such a family with $T$ connected.
\item The \emph{Picard group} $\operatorname{Pic}(X)$ of a $K$-variety $X$ is the group of isomorphism classes of line bundles on $X$, and $\operatorname{Pic}^0(X)$ is the subgroup of those line bundles that are algebraically equivalent to the trivial line bundle.
\end{enumerate}
\end{definition}

\begin{corollary}\label{cor:ActionIntegersLineBundles}
The following formula holds for all line bundles $\mathcal{L}$ and all integers $n$, where $[n]$ denotes the multiplication by $n$ map $[n]: A \to A$ (see Definition \ref{def:MultiplicationByN} below):
\[
[n]^* \mathcal{L} \cong \mathcal{L}^{\otimes \frac{n^2+n}{2}} \otimes [-1]^* \mathcal{L}^{\otimes \frac{n^2-n}{2}}
\]
Furthermore, if $[\mathcal{L}] \in \operatorname{Pic}^0(A)$, then $[-1]^*\mathcal{L} \cong \mathcal{L}^{-1}$, so that $[n]^* \mathcal{L} \cong \mathcal{L}^{\otimes n}$.
\end{corollary}

\subsection{Morphisms of abelian varieties; isogenies}

Abelian varieties form a category, with morphisms defined as follows.
\begin{definition}\label{def:homs and isogenies}
Let $A$ and $B$ be abelian varieties over the same field $K$.
\begin{enumerate}
\item A \emph{homomorphism} of abelian varieties $f \colon A \to B$ is a morphism of $K$-group schemes. When we want to stress the fact that $f$ is a morphism of varieties over $K$, we will say that $f$ is defined over $K$. Equivalently, a homomorphism of abelian varieties $f \colon A \to B$ is a morphism of varieties that takes $0_A$ to $0_B$ and such that $f(a_1 +_A a_2) = f(a_1) +_B f(a_2)$ for all $a_1, a_2 \in A(\overline{K})$. 
\item  An \emph{isogeny} is a homomorphism of abelian varieties $f \colon A \to B$ that is both surjective and with finite kernel. 
\end{enumerate}
\end{definition}

\begin{definition}
Let $A, B$ be two $K$-abelian varieties and let $L$ be an extension of $K$.
\begin{enumerate}
\item We let $\Hom_L(A,B)$ (often also denoted by $\Hom(A_L,B_L)$) be the set of homomorphisms of $L$-abelian varieties $A_L \to B_L$. We write simply $\Hom(A, B)$ for $\Hom_K(A, B)$.

Note that $\Hom_L(A,B)$ has a natural structure of group, using addition on the target as the group operation.
\item We let $\End_L(A)$ (also denoted by $\End(A_L)$) be the set of homomorphisms of $L$-abelian varieties from $A_L$ to itself. We write $\End(A)$ for $\End_K(A)$.

Note that $\End_L(A)$ is a ring, with addition given by addition on the target and multiplication given by composition.
\end{enumerate}
\end{definition}

A key example of an isogeny is the multiplication-by-$n$ map that we now introduce formally.

\begin{definition}[Canonical copy of $\Z$ inside $\End_K(A)$]\label{def:MultiplicationByN}
The ring $\End_K(A)$ contains a canonical copy of the integers: indeed, for every $n \in \N$ the map
\[
\begin{array}{cccc}
[n] \colon & A &  \to &  A \\
& x & \mapsto & \underbrace{x + \cdots + x}_{n \text{ times}}
\end{array}
\]
is an endomorphism of $A$. We further define $[-1]$ to be the map giving the inverse for the group law, and for $n>1$ we define $[-n]$ as the composition of $[n]$ and $[-1]$. This gives a canonical identification $n \mapsto [n]$ of $\Z$ with a subring of $\End_K(A)$. One often says that $A$ has \emph{trivial endomorphisms over $K$} if $n \mapsto [n]$ induces an isomorphism $\Z \cong \End_K(A)$.
\end{definition}

A very general abelian variety $A$ over $\C$ has only trivial endomorphisms, i.e., $\End_{\C}(A)\cong\Z$. As an explicit example of an abelian variety with endomorphism ring $\Z$, one can take the elliptic curve with Weierstrass equation $y^2=x^3+x+1$. We give an example of an abelian variety possessing non-trivial endomorphisms.

\begin{example}
Consider the elliptic curve $E \colon y^2=x^3+x$ over $\Q(i)$. There is an endomorphism (in fact, automorphism) $[i] \colon E \to E$ given by $(x,y) \mapsto (-x, iy)$. One can show that $\End(E)$ is isomorphic to $\Z[i]$, where the copy of $\Z$ inside $\End(E)$ is given by the multiplication maps $[n]$, and $i \in \Z[i]$ is identified with the automorphism $[i] \in \End(E)$.
\end{example}

The following important consequence of Proposition \ref{prop: AV are commutative} explains why -- to study arbitrary morphisms between two abelian varieties -- the crucial step is to understand the homomorphisms between them.

\begin{proposition}\label{prop:MorAV}
Let $f \colon A \to B$ be an arbitrary morphism of abelian varieties. There exist a homomorphism $h \colon A \to B$ and a translation $\tau \colon B \to B$ such that $f = \tau \circ h$ (that is, $f$ is the composition of a group homomorphism with a translation on the target).
\end{proposition}
\begin{proof}
Replacing $f$ with $\tilde{f}(x) \coloneq f(x) - f(0)$, we are reduced to showing that a morphism $f \colon A \to B$ such that $f(0_A)=0_B$ is a homomorphism of abelian varieties. Consider the map 
\[
\begin{array}{cccc}
\varphi: & A \times A & \to & B \\
& (a_1,a_2) & \mapsto & f(a_1)+f(a_2)-f(a_1+a_2)\,.
\end{array}
\]
By the rigidity lemma (Lemma \ref{lemma:Rigidity}), noticing that $\varphi(\{0_A\} \times A)=\varphi(A \times \{0_A\}) = \{0_B\}$, we obtain $\varphi(A \times A)=\{0_B\}$, that is, $f(a_1)+f(a_2)=f(a_1+a_2)$, as desired.
\end{proof}


An important invariant of isogenies is their degree, which we now define.

\begin{definition}
Let $f\colon X \to Y$ be a finite surjective morphism between algebraic varieties over a field $K$. The \emph{degree} of $f$ is the degree of the finite field extension of the function field $K(X)$ over $f^*K(Y)$.
\end{definition}

\begin{definition}
 The \emph{degree} of an isogeny $\varphi$ is its degree in the sense of the previous definition (see Proposition \ref{prop:CharacterisationIsogenies} below). The degree of a homomorphism $f \colon A \to B$ is either its degree as an isogeny, or $0$ if $f$ is not an isogeny.
\end{definition}

\begin{remark}
The degree of an isogeny $\varphi$ agrees with the order of $\ker \varphi$, where \emph{order} means \emph{rank as a finite group scheme}. When the isogeny is separable (which is always the case in characteristic zero), the order of $\ker \phi$ is the cardinality of $(\ker \phi)(\overline{K})$. This is also the cardinality of the finite set $\varphi^{-1}(0_B)$, where we interpret $\varphi$ as a map $\varphi \colon A(\overline{K}) \to B(\overline{K})$.
\end{remark}

It is useful to know that isogenies can be characterised in many equivalent ways: the following standard result can be found for example in \cite[Proposition 8.1]{MilneAV}.

\begin{proposition}\label{prop:CharacterisationIsogenies} For a homomorphism $f \colon A \to B$ of abelian varieties, the following statements are equivalent:
\begin{enumerate}
\item $f$ is an isogeny;
\item  $\dim A = \dim B$ and $f$ is surjective;
\item $\dim A = \dim B$ and $\ker f$ is a finite group (scheme);
\item $f$ is finite, flat, and surjective.
\end{enumerate}
\end{proposition}

\begin{remark}
Notice in particular that a $K$-isogeny $\varphi\colon A \to B$ can only exist if $A$ and $B$ have the same dimension, see also Definition \ref{def:IsogenousAV} and Remark \ref{rmk:IsogenySymmetric}.
\end{remark}

\begin{remark}
It is clear that the degree is multiplicative: if $f \colon A \to B$ and $h \colon B \to C$ are isogenies, we have
\[
\deg(h \circ f)=\deg(h) \deg(f).
\]
Moreover, a homomorphism $f \colon A \to A$ that is not an isogeny cannot be surjective, and the composition of two endomorphisms, at least one of which fails to be surjective, cannot be surjective. This implies that $\deg \colon \End_K(A) \to \N$ satisfies
\[
\deg(fh)=\deg(f)\deg(h)
\]
for all $f, h \in \End_K(A)$.
\end{remark}


\begin{definition}[Isogenous abelian varieties]\label{def:IsogenousAV}
We say that $A$ is isogenous to $B$ (over $K$), and write $A \sim B$, if there exists an isogeny $\varphi \colon A \to B$ (defined over $K$).
\end{definition}

\begin{remark}[Isogeny is an equivalence relation]\label{rmk:IsogenySymmetric}
The condition of being isogenous is an equivalence relation. Indeed, it is clear that $\sim$ is reflexive and transitive, and it suffices to check that it is symmetric. Suppose that $\varphi\colon A \to B$ is an isogeny: we want to construct an isogeny $B \to A$ in the opposite direction. Since $\ker \varphi$ is a finite group (scheme), it is in particular of finite exponent $N$, so $\ker \varphi \subseteq \ker [N]$ (also as group schemes). Consider the following commutative diagram:
\[
\begin{tikzcd}[row sep=3.5em, column sep=4.0em]
A \arrow[d, "\pi_1"'] 
  \arrow[rr, "{[N]}"] 
  \arrow[rrdd, "\pi_2"] 
  \arrow[dd, "\varphi"', bend right=45] 
& & A \\
A/\ker\varphi \arrow[d, dashed, "\psi"'] \arrow[rrd, "\pi"'] & & \\
B \arrow[rr, dashed, "\chi"'] & & A/\ker[N] \arrow[uu, "\omega"]
\end{tikzcd}
\]
We notice that $\psi$ exists and is an isomorphism because of the universal property of the quotient $A \to A/\ker \varphi$; moreover, since $\ker \varphi$ is contained in $\ker [N]$, we obtain the existence of a homomorphism $\pi$ as in the diagram (in fact, $\pi$ is nothing but the canonical projection from $A/\ker \varphi$ to $\frac{A/\ker \varphi}{A[N]/\ker \varphi}$). In particular, we may define $\chi \coloneq  \pi \circ \psi^{-1}$. Finally, by using again the universal property of the quotient, we conclude that $\omega$ is an isomorphism. Putting everything together we may define a homomorphism $B \to A$ as the composition $\omega \circ \chi = \omega \circ \pi \circ \psi^{-1}$; this homomorphism is in fact an isogeny because $\omega$ and $\psi^{-1}$ are isomorphisms and $\pi$ is an isogeny.
\end{remark}

The argument in the previous remark implies in particular:
\begin{proposition}\label{prop:FactorThroughDegree}
Let $f \colon A \to B$ be an isogeny of degree $d$. There exists an isogeny $h \colon  B \to A$ such that $h \circ f=[d]_A$ and $f \circ h=[d]_B$.
\end{proposition}
\begin{remark}\label{rmk:EtalenessIsogenies}
It is often useful to think about isogenies in geometric/topological terms: over $\C$, for example, an isogeny $\varphi \colon A \to B$ induces a covering map $\varphi \colon A(\C) \to B(\C)$, which is Galois with Galois group isomorphic to $\ker \varphi$. More generally, over a field of characteristic $0$, isogenies are étale morphisms (recall that \emph{étale morphism} is the algebro-geometric version of \emph{covering map}). Even more generally, an isogeny of degree $n$ is étale over any field $K$ such that $(n, \operatorname{char} K)=1$. The reason for this is that the group structure allows one to carry étaleness from one point to another, and étaleness at zero follows from the fact that the multiplication-by-$n$ map $[n]$ induces multiplication by $n$ on the tangent space at $0$ of any abelian variety, together with the fact that any isogeny of degree $n$ is a factor of $[n]$ (by Proposition \ref{prop:FactorThroughDegree}).
\end{remark}

We can compute the degree of $[n]$ in general, as a function of $n$ and $g=\dim A$.
\begin{proposition}
For every positive integer $n$, the degree of $[n]$ is $n^{2g}$.
\end{proposition}
\begin{proof}
In order to compute the degree of $[n]$ we look at its action on a very ample line bundle $\mathcal{L}$. 
One can always find a \emph{symmetric} ample line bundle, namely an ample $\mathcal{L}$ such that $[-1]^* \mathcal{L} \cong \mathcal{L}$: indeed, if $\mathcal{M}$ is any ample line bundle, $\mathcal{L} \coloneq  \mathcal{M} \otimes [-1]^* \mathcal{M}$ is ample and symmetric. Taking a sufficiently large power will make it very ample; we still call $\mathcal{L}$ the resulting line bundle.

Corollary \ref{cor:ActionIntegersLineBundles} implies that $[n]^* \mathcal{L} \cong \mathcal{L}^{\otimes n^2}$ -- notice that this is in particular compatible with the assumption $[-1]^* \mathcal{L} \cong \mathcal{L}$. It follows that $[n]^* \mathcal{L}|_{\ker[n]} \cong \mathcal{L}^{\otimes n^2} |_{\ker [n]}$ is both trivial and very ample, which is only possible if $\ker [n]$ is zero-dimensional (hence $[n]$ is an isogeny).
Furthermore, writing $\mathcal{L}$ as $\mathcal{O}(D)$, we have by the projection formula
\[
\deg [n] \left( D, \dotsc, D\right) = \left([n]^* D, \dotsc, [n]^* D\right) = \bigl(n^2D , \dotsc, n^2 D\bigr) = n^{2g}\left( D, \dotsc, D\right)
\]
where $\left( D,\dotsc,D \right)$ denotes the $g$-fold intersection product. To conclude, it suffices to show that $\left( D, \dotsc, D\right)$ is non-zero, but this is easy, because since $\mathcal{L}$ is very ample (hence it induces an embedding $A \hookrightarrow \bbP^N$) we may compute this intersection product as the intersection product of $g$ general hyperplane sections of $A \hookrightarrow \bbP^N$, and this is clearly positive.
\end{proof}

\begin{definition}[Abelian subvarieties]
Let $A/K$ be an abelian variety. An \emph{abelian subvariety} is a \emph{smooth connected closed} subgroup scheme $B \subseteq A$, that is, a subvariety of $A$ that is itself an abelian variety with the induced operations. An abelian variety $A/K$ is said to be \emph{simple} (or $K$-simple, if we want to stress the field of definition) if the only abelian subvarieties of $A$ (defined over $K$) are $A$ itself and $\{0_A\}$. When $A \times_K \overline{K}$ is simple, one says that $A$ is \emph{geometrically simple} or \emph{absolutely simple}: clearly, absolutely simple implies simple, but the converse implication does not hold. 
\end{definition}

\begin{definition}\label{def:torsion AV 1}
    Let $A$ be an abelian variety over a field $K$ and $n$ be a positive integer. We define $A[n]$ as the kernel of $[n] \colon A \to A$. When $n$ is not divisible by the characteristic of $K$, we can and will identify $A[n]$ with the finite abstract group $\{P \in A(\overline{K}) \colon nP = 0 \}$, together with the natural action of $\Gal(\overline{K}/K)$ on it. (Note that Remark \ref{rmk:EtalenessIsogenies} shows that $A[n]$ is an étale group scheme when $n$ is not divisible by the characteristic of $K$, and Remark \ref{rmk: etale group schemes} then gives the above identification.)
\end{definition}

We call $P\in A(\overline{K})$ a \emph{torsion point} if there exists some positive integer $n$ such that $nP=0_A$. We will investigate the torsion of abelian varieties in much greater detail in Section \ref{subsec: torsion and Tate module}.

\subsection{Poincaré's total reducibility theorem}

The following result, due to Poincaré, shows that every abelian variety can be decomposed as a product of simple ones up to isogeny.
\begin{theorem}[{\cite[Theorem 12.2]{EvGM}}]\label{thm:Poincare}
Let $A/K$ be an abelian variety. The following holds:
\begin{enumerate}
\item  If $B$ is an abelian subvariety of $A$ defined over $K$, then there exists an abelian subvariety $C$ of $A$, also defined over $K$, such that
\[
\begin{array}{ccc}
B \times C & \to & A \\
(b,c) & \mapsto & b + c
\end{array}
\]
is an isogeny. The subvariety $C$ is often called a \emph{complement} to $B$ in $A$.
\item There exist $K$-simple abelian $K$-subvarieties $A_1, \dotsc, A_n$ of $A$ such that
\[
\begin{array}{ccc}
A_1 \times \cdots \times A_n & \to & A \\
(a_1,\dotsc,a_n) & \mapsto & a_1 + \cdots + a_n
\end{array}
\]
is an isogeny.
\end{enumerate}
\end{theorem}


\begin{remark}
The isogeny in Poincaré's theorem is usually \emph{not} an isomorphism. When $n=2$ (i.e., there are two simple subvarieties $A_1, A_2$ of $A$ such that the sum $+ \colon A_1 \times A_2 \to A$ is an isogeny), the kernel of the sum is essentially the intersection $A_1 \cap A_2$.
\end{remark}

\begin{remark}
It is sometimes useful to work in the \emph{category $\mathcal{S}$ of abelian varieties (over $K$) up to isogeny}. This is the category whose objects are abelian varieties over $K$ and such that $\Hom_{\mathcal{S}}(A,B)= \Hom_{K}(A,B) \otimes_{\Z} \Q$. Remark \ref{rmk:IsogenySymmetric}, together with Poincaré's complete reducibility theorem (see Theorem \ref{thm:Poincare}), implies that $\mathcal{S}$ is a \emph{semisimple category}: every object is a direct sum of simple objects. In less fancy language, this means that every abelian variety is isogenous, although not necessarily isomorphic, to the direct product of simple abelian varieties.
\end{remark}
\subsection{Abelian varieties built from elliptic curves}
An important source of examples is given by \emph{Jacobians}, a special class of abelian varieties canonically associated with curves. We will introduce Jacobians soon. For now, we can only describe very basic examples of abelian varieties.
\begin{example}\phantom{~}
\begin{enumerate}
\item Elliptic curves are abelian varieties. In fact, the term \emph{elliptic curve} is synonymous with \emph{abelian variety of dimension 1}. Recall that an elliptic curve is a nice genus 1 curve \emph{with a marked rational point}: the rational point is essential in defining the group law. In fact, Proposition \ref{prop:MorAV} implies in particular that the group law is uniquely determined by the choice of the neutral element.
\item Let $E_1, \dotsc, E_g$ be elliptic curves over the field $K$: then $E_1 \times \cdots \times E_g$ is a group scheme over $K$ which is connected, smooth and projective, hence a $K$-abelian variety of dimension $g$. One can prove that not all $g$-dimensional abelian varieties are of this form -- in fact, products of elliptic curves are (in a precise sense) very sparse among $g$-dimensional abelian varieties. 
\end{enumerate}
\end{example}

\subsection{Abelian varieties over \texorpdfstring{$\C$}{C}}
The theory of complex abelian varieties -- that is, abelian varieties over the complex numbers -- is already very rich, and the existence of the \emph{analytic uniformisation} (see Theorem \ref{thm:AnalyticUniformisation} below) makes it more intuitive than the theory over general fields, so we start by discussing this case in detail. 
Let $A/\C$ be an abelian variety. Notice that $A(\C)$ is a compact complex manifold of dimension $g$ endowed with a group structure compatible with the differential structure; in other words, it is a compact complex Lie group of dimension $g$. We now show that -- from the point of view of differential geometry -- this group has a very simple form:

\begin{theorem}[Analytic uniformisation of abelian varieties over $\mathbb{C}$]\label{thm:AnalyticUniformisation}
For any $g$-dimensional abelian variety $A$ over the complex numbers, there exists a lattice $\Lambda \subseteq \C^g$ such that 
$A(\C)$ is isomorphic (as a Lie group) to $\C^g/\Lambda$.
\end{theorem}

\begin{proof}
Write $V = T_0A$ for the tangent space at the identity. From the differential geometry of Lie groups we know that for every $v \in V$
there is a unique analytic group homomorphism
\[
\varphi_v \colon \C \to A
\]
with $d\varphi_v(0) = v$. One also knows that $\varphi \colon V \times \C\to A$ is analytic. The \emph{exponential
map} is
\[
\begin{array}{cccc}
\exp\colon & V & \to & A \\
& v & \mapsto & \varphi_v(1).
\end{array}
\]
It is clear that $\varphi_v(t)=\exp(tv)$: by the uniqueness of $\varphi_v$ we have $\varphi_v(st) = \varphi_{tv}(s)$, so the formula follows by setting $s=1$. Moreover, $\exp: V \to A$ is a group homomorphism: indeed
\[
t \mapsto \exp(tx)+\exp(ty)
\]
is a group homomorphism (because $A$ is abelian); taking the derivative at $0$ and using the uniqueness of $\varphi_{x+y}$ we obtain $\exp(tx)+\exp(ty)=\exp(t(x+y))$ as claimed. Finally, $\exp$ is surjective because $\exp(V)$ is a subgroup of $A(\C)$ that contains a neighbourhood of the identity, and $A$ is connected.

Now define $\Lambda$ to be the kernel of $\exp$: on the one hand, $\Lambda$ is a discrete subgroup of $\C^g$, because $\exp$ is a local homeomorphism; on the other hand, $\Lambda$ must have full rank since $\C^g/\Lambda \cong A$ is compact.
\end{proof}

\begin{definition}[Complex tori] A \emph{complex torus} is any complex analytic variety of the form $\C^g/\Lambda$ for some $g \geq 1$ and for some full-rank lattice $\Lambda$.
\end{definition}

\begin{remark}
It is \emph{not} true that every complex torus is an abelian variety. The precise conditions under which this happens are known as \emph{Riemann relations}:
the obstruction is that a general complex torus does \emph{not} admit an analytic embedding in projective space.
\end{remark}

\begin{remark}
Despite the previous remark, every complex torus of dimension 1 is an abelian variety (hence in particular an elliptic curve: the marked point is given by the class in $\C/\Lambda$ of the zero vector in $\C$). The Riemann conditions are automatically satisfied for $1$-dimensional tori: see Example \ref{ex:CanonicalPolarisationEC} below.
\end{remark}

\begin{proposition}[Torsion points of complex abelian varieties]\label{prop:TorsionCpxAV} Let $A$ be a $g$-dimensional complex abelian variety and let $n$ be a positive integer. The group
\[
A[n] = \{ x \in A(\C) \colon nx=\underbrace{x+ \cdots +x}_{n \text{ times}} = 0 \}
\]
is isomorphic to $(\Z/n\Z)^{2g}$.
\end{proposition}
\begin{proof}
By analytic uniformisation it suffices to understand the $n$-torsion points of the group $\C^g/\Lambda \cong \R^{2g}/\Lambda$.
\nomenclature{$\R$}{the field of real numbers}
As an abstract group, this is isomorphic to $\R^{2g}/\Z^{2g}$, because up to a change of basis in $\R^{2g}$ we can assume that $\Lambda$ is the standard lattice $\Z^{2g}$.

It follows that the group of $n$-torsion points of $A$ is isomorphic to
\[
\left(\R/\Z\right)^{2g} \![n] \cong \bigl((\R/\Z)[n]\bigr)^{2g} \cong (\Z/n\Z)^{2g},
\]
where we have denoted by $G[n]$ the $n$-torsion points of an abstract abelian group $G$.
\end{proof}
\begin{remark}
Implicit in the above proof is the fact that any complex abelian variety is isomorphic to $(\bbS^1)^{2g}$ as a topological space, and in fact also as a real analytic manifold. Here $\bbS^1 = \{ z \in \C^\times \colon |z|=1\}$ denotes the multiplicative group of complex numbers of modulus $1$. All the richness of the theory comes from the complex structure.
\end{remark}

We now come to the existence of an embedding in projective space. In the interest of concreteness, we describe only one of the many possible definitions of a polarisation and of the dual abelian variety:

\begin{definition}[Dual abelian variety, analytic setting]\label{def:DualAVC}
Let $V=\C^g$ and write $A=V/\Lambda$. Let $\overline{V}^\vee$ be the space of $\C$-antilinear functionals $V \to \C$. The vector space $\overline{V}^\vee$ contains a natural lattice, namely
\[
\Lambda^\vee \coloneq \bigl\{ \psi \in \overline{V}^\vee \colon \Im \psi(\Lambda)\subseteq \Z \bigr\},
\]
where $\Im$ denotes the imaginary part. One can check that the rank of $\Lambda^\vee$ is maximal, so that $\overline{V}^\vee/\Lambda^\vee$ is again an abelian variety, called the \emph{dual abelian variety} of $A$.
\end{definition}

\begin{definition}[Polarisation, analytic setting]\label{def:AnalyticPolarisation}
We keep the notation of Definition \ref{def:DualAVC}. 
Let $H \colon V \times V \to \C$ be a Hermitian form (linear in the first argument, antilinear in the second). We say that $H$ is a \emph{polarisation} if it is positive-definite and $\Im H|_{\Lambda \times \Lambda}$ is integer-valued. The imaginary part of $H$ is alternating, whereas its real part is symmetric. The integrality condition ensures that the antilinear functional $H(u,\cdot)$ belongs to the dual lattice $\Lambda^\vee$ for every $u\in\Lambda$, and hence that $v\mapsto H(v,\cdot)$ induces a homomorphism $A\to A^\vee$.
\end{definition}

\begin{remark}
We note that the word \emph{polarisation} can be defined in very different ways in different sources: an ample divisor on $A$, a certain bilinear form, or an isogeny from $A$ to $A^\vee$ (see below) might all be reasonably called polarisations. We will not describe the equivalence between the various notions in detail: the interested reader may consult, for example, the book by Birkenhake and Lange \cite[§4.1]{BirkenhakeLangeAV}, which gives a very clear picture of the situation over the complex numbers.
\end{remark}

\begin{remark}\label{rmk:RealDefAVC}
A complex torus $\C^g/\Lambda$ is a complex abelian variety if and only if it admits at least one polarisation.
\end{remark}

\begin{definition}[Type of a polarisation]
Linear algebra over $\Z$ (essentially the elementary divisors theorem) shows that in a suitable basis of $\Lambda$ the matrix representation of $\Im H|_{\Lambda \times \Lambda}$ takes the form
\[
\begin{pmatrix}
0 & D \\
-D & 0
\end{pmatrix}, \quad \text{ where } \quad D = \begin{pmatrix}
d_1 \\
& d_2 \\
& & \ddots \\
& & & d_g
\end{pmatrix},
\]
and where the $d_i$ are positive integers such that $d_1 \mid d_2 \mid \cdots \mid d_g$. The vector $(d_1,\dotsc,d_g)$ is called the \emph{type} of the polarisation; one also sets $d(H)=\prod d_i$. 
\end{definition}
%

\begin{remark}
A polarisation induces a map
\[
\begin{array}{cccc}
\lambda_H \colon & V &\to & \overline{V}^\vee \\
& v  & \mapsto & H(v,\cdot)
\end{array}
\]
which is surjective and satisfies $\lambda_H(\Lambda) \subseteq \Lambda^\vee$. In particular, it induces a surjective analytic homomorphism $\lambda_H \colon V/\Lambda \to \overline{V}^\vee/\Lambda^\vee$ which is easily seen to have finite kernel (hence is an isogeny). The polarisation $H$ is said to be \emph{principal} if $\lambda_H$ is an isomorphism; this happens precisely when $d(H)=1$.
\end{remark}

\begin{example}[Canonical polarisation of an elliptic curve]\label{ex:CanonicalPolarisationEC}
We now show that every elliptic curve over $\C$ admits a canonical principal polarisation. Up to a $\C$-linear change of variables, we may assume that the lattice $\Lambda$ is $\Z \cdot 1 \oplus \Z \cdot \tau$, with $\tau$ in the upper half-plane. To define a polarisation we need to find a positive-definite Hermitian form on $\C \times \C$ whose imaginary part takes integer values on $1$ and $\tau$. A Hermitian product on $\C$ is uniquely defined by its value on $(1,1)$, so we look for a polarisation of the form
\[
\begin{array}{cccc}
H \colon & \C \times \C & \to & \C \\
& (z_1,z_2) & \mapsto & \gamma z_1\overline{z_2};
\end{array}
\]
in order for $H$ to be Hermitian and positive definite, $\gamma$ needs to be real and positive.

Write $E\coloneq \Im H$; one checks that $E$ is skew-symmetric (this is always true: if $H$ is a polarisation, $E=\Im H$ is skew-symmetric), hence we have $E(1,1)=E(\tau,\tau)=0$; the only requirement is then $E(\tau,1) \in \Z$, that is, $\gamma \Im \tau \in \Z$. Since $\Im \tau>0$, we can choose $\gamma \coloneq  \frac{1}{\Im \tau}$. We then obtain the polarisation
\[
H(z_1,z_2) = \frac{z_1 \overline{z_2}}{\Im \tau}.
\]
It is also easy to see that $H$ is a principal polarisation (and in fact it's the \emph{unique} principal polarisation of $E$). Assuming the previous remark, it would be enough to notice that in the basis $\{\tau,1\}$ the matrix of $\Im H$ is $\begin{pmatrix}
0 & 1 \\ -1 & 0
\end{pmatrix}$ to deduce that $d(H)=1$ and that $H$ is a principal polarisation. We check this by hand by finding the lattice $\Lambda^\vee$. An antilinear functional $\psi$ lies in $\Lambda^\vee$ if and only if $\Im \psi(1) \in \Z$ and $\Im \psi(\tau) \in \Z$. Clearly $\psi$ is determined by $\psi(1)=a+bi$ with $b \in \Z$; writing $\tau=\Re \tau +i\Im \tau$ we have
\[
\begin{aligned}
\psi(\tau) & = \psi(\Re \tau +i \Im\tau) = \Re\tau \psi(1) - i \Im\tau \psi(1) \\
& = \Re\tau(a+bi)-i \Im\tau(a+bi) = (a\Re\tau +b \Im\tau) + i (b \Re\tau - a \Im\tau),
\end{aligned}
\]
so that $\Im \psi(\tau)=b\Re\tau-a\Im\tau$. By assumption, this should be equal to an integer, say $n$; this happens if and only if
\[
a = \frac{b\Re \tau -n}{\Im \tau}.
\]
Hence $\displaystyle \psi(1) = a+bi = \frac{b \Re \tau + b i \Im \tau - n}{\Im \tau} = \frac{b \tau - n}{\Im \tau}$, and therefore $\psi$ is of the form $H(\lambda, \cdot)$, where $\lambda=b\tau-n \in \Lambda$ is an element of $\Lambda$. This implies easily that $\lambda_H \colon V \to \overline{V}^\vee$ induces an isomorphism of $\Lambda$ with $\Lambda^\vee$, hence an isomorphism $\lambda_H \colon V/\Lambda \cong \overline{V}^\vee/\Lambda^\vee$ as claimed.
Finally, we remark that it is a general fact that an elliptic curve over any field $K$ admits precisely one polarisation of degree $d^2$ for every $d \geq 1$ (see e.g.~\cite{ConradPolarisations}). Under the identification $E\simeq E^\vee$ given by the canonical principal polarisation $H$, the endomorphism induced by $z\mapsto\alpha z$ (where $\alpha \in \mathbb{C}$ satisfies $\alpha \Lambda \subseteq \Lambda$) corresponds to the form $H_\alpha(z_1,z_2)=H(\alpha z_1,z_2)=\alpha H(z_1,z_2)$. This form is Hermitian and positive-definite if and only if $\alpha$ is real and positive. Since the endomorphism ring of $E$ -- equivalently, the set of possible $\alpha$ -- is either $\mathbb{Z}$ (in the non-CM case) or an order in an imaginary quadratic field (in the CM case), we see that $H_\alpha$ is a polarisation if and only if $\alpha$ is a positive integer. Thus, the only polarisations are the positive integer multiples of $H$, with the degree of $dH$ being $d^2$.
\end{example}

We conclude this section with a result valid over any algebraically closed field, but easier to prove over $\C$:
\begin{theorem}\label{thm:IsogenyToPPAVOverC}
Every abelian variety $A/\C$ is $\C$-isogenous to a principally polarisable abelian variety, that is, there is a surjective analytic homomorphism $\varphi$ from $A$ to a principally polarisable abelian variety $A'$ such that $\varphi$ has finite kernel.
\end{theorem}
\begin{proof}
Write $A=\C^g/\Lambda$, choose any polarisation $H$ (at least one exists, because $A$ is an abelian variety and not just a complex torus), and fix a $\Z$-basis $\tau_1,\dotsc, \tau_{g}, \tau_{g+1},\dotsc,\tau_{2g}$ of $\Lambda$ such that the matrix of $\Im H$ is of the form
\[
\Im H = \begin{pmatrix}
0 & D \\
-D & 0
\end{pmatrix}
\]
with $D=\operatorname{diag}(d_1,\dotsc,d_g)$. Now consider the lattice 
\[
\Lambda' \coloneq \bigoplus_{i=1}^g  \Z \cdot \frac{1}{d_i}\tau_i \oplus \bigoplus_{i=1}^g \Z \cdot \tau_{g+i}\,.
\]
It is immediate to see that $\Lambda'$ is an over-lattice of $\Lambda$, and that with respect to the obvious basis $\{\frac{1}{d_i} \tau_i, \tau_{g+i}\}_{i=1,\dotsc,g}$ of $\Lambda'$ the matrix representing $\Im H$ is $\begin{pmatrix}
0 & \operatorname{I}_g \\
-\operatorname{I}_{g} & 0
\end{pmatrix}$, so that the abelian variety $A'=\C^g/\Lambda'$ is principally polarised. Now simply observe that there is an isogeny $A=\C^g/\Lambda \to \C^g/\Lambda'=A'$ induced by the identity of $\C^g$ (the kernel is the finite group $\Lambda'/\Lambda$).
\end{proof}
We note that there exist abelian varieties that do not admit principal polarisations: see Example \ref{ex:no-principal-polarisation}.

\subsection{The dual abelian variety, polarisations, and the Weil pairing}\label{sect:DualAV}

We now introduce the notion of the \emph{dual abelian variety} of an abelian variety $A$ over an arbitrary field. If the reader is familiar only with elliptic curves, this concept may not yet have appeared explicitly: indeed, when $A=E$ is an elliptic curve, the dual abelian variety is canonically isomorphic to $E$ itself, so the distinction is usually suppressed. For abelian varieties of dimension greater than one, however, $A$ and its dual are in general not isomorphic, and it becomes essential to distinguish the two.

Roughly speaking, the dual abelian variety parametrises algebraically trivial line bundles on $A$. We make this precise in the following definition, where we use the notion of the \emph{Picard group} $\operatorname{Pic}(A)$ and its subgroup $\operatorname{Pic}^0(A)$ (see Definition \ref{def:PicardGroup} and Definition \ref{def: Picard group curve} for the special case of curves).


\begin{definition}[Dual abelian variety]
Let $A$ be an abelian variety over $K$. A pair $(A^\vee,\mathcal{P})$, where $A^\vee$ is an abelian variety over $K$ and $\mathcal{P}$ is an invertible sheaf (line bundle) on $A \times_K A^\vee$, is called a (the) dual abelian variety of $A$ if the following conditions hold:

\begin{enumerate}
\item For every $b \in A^\vee(\overline{K})$, the restriction $\mathcal{P}|_{A \times \{b\}}$ lies in $\Pic^0(A_{\overline{K}})$;
\item The restriction $\mathcal{P}|_{\{0\} \times A^\vee}$ is trivial;
\item $(A^\vee,\mathcal{P})$ is universal among such pairs, that is, the following universal property holds: for all pairs $(T,\mathcal{L})$ consisting of a $K$-scheme $T$ and an invertible sheaf $\mathcal{L}$ on $A \times T$ that satisfies
\begin{enumerate}
\item[(a)] $\mathcal{L}|_{A \times \{t\}}$ is in $\Pic^0(A_t)$ for all $t \in T$
\item[(b)] $\mathcal{L}|_{\{0\} \times T}$ is trivial
\end{enumerate}
there is a unique regular map $\alpha: T \to A^\vee$ such that $\mathcal{L} \cong (1 \times \alpha)^* \mathcal{P}$. The line bundle $\mathcal{P}$ is then called the \emph{Poincaré line bundle}.
\end{enumerate}
\end{definition}

This definition makes precise that the dual abelian variety $A^\vee$ represents the functor $T \mapsto \Pic^0(A \times_K T / T)$, i.e., the relative Picard functor restricted to algebraically trivial line bundles. The Poincaré bundle $\mathcal{P}$ serves as the universal family of such line bundles.

\begin{remark}
Of course, one should show that the dual abelian variety \emph{exists}! This is done in \cite{Mumford70abelianvarieties}.
Moreover, the assignment $A \mapsto A^\vee$ is functorial and defines a duality: for every $A$, there is a canonical isomorphism $(A^\vee)^\vee \cong A$.
\end{remark}

\begin{remark}
An equivalent way of stating the universal property is that there is a bijection
\[
\operatorname{Mor}(T,A^\vee) \leftrightarrow \left\{ \begin{array}{c}\text{line bundles } \mathcal{L} \text{ on } A \times T \\ \text{ satisfying (a), (b)} \end{array}\right\} / \sim,
\]
where the set on the left denotes the set of morphisms of $K$-schemes $T \to A^\vee$ (the scheme $T$ need not carry a group structure) and $\sim$ is isomorphism of line bundles. Applying this characterisation to $T=\operatorname{Spec}(K)$ we obtain $A^\vee(K) = \Pic^0(A)$.
\end{remark}

We describe a standard way to construct maps $A \to A^\vee$:
\begin{definition}[Mumford's construction]
\label{def:MumfordsConstruction}
Let $\mathcal{L}$ be a line bundle on $A_{\overline{K}}$. We define
\[
\begin{array}{cccc}
\lambda_{\mathcal{L}} \colon & A_{\overline{K}} & \to & A^\vee_{\overline{K}} \\
& a & \mapsto & \tau_a^* \mathcal{L} \otimes \mathcal{L}^{-1},
\end{array}\]
where $\tau_a$ is translation by $a$. It follows from the theorem of the square (Theorem \ref{thm:Square}) that $\lambda_{\mathcal{L}}$ is a homomorphism.
\end{definition}

\begin{remark}
The line bundles $\mathcal{L}$ in $\Pic^0(A_{\overline{K}})$ are precisely those for which $\lambda_{\mathcal{L}}$ is the zero map.
\end{remark}

We have already encountered polarisations in the context of complex abelian varieties (see Definition \ref{def:AnalyticPolarisation}): we now introduce their algebraic counterparts.

\begin{definition}[Polarisation, algebraic setting]
A $K$-\emph{polarisation} of the abelian variety $A/K$ is a $K$-isogeny $\varphi \colon A \to A^\vee$ such that \emph{the base-change to $\overline{K}$} of $\varphi$ is of the form $\lambda_{\mathcal{L}}$ for some ample line bundle $\mathcal{L}$. (This is not quite the same as requiring that $\varphi$ is of the form $\lambda_{\mathcal{L}}$ already over $K$.)

A \emph{principal polarisation} is a polarisation which induces an isomorphism $A \cong A^\vee$; notice that principal polarisations need not exist if $\dim A >1$.
A pair $(A,\lambda)$, where $A$ is an abelian variety and $\lambda$ is a principal polarisation, is usually called a \emph{principally polarised abelian variety}, or \emph{PPAV} for short.
\end{definition}

\begin{remark}\label{rmk:PolarisationsExist}
Let $A/K$ be an abelian variety over an arbitrary field. Then $A$ and $A^\vee$ are isogenous (but not necessarily isomorphic) over $K$. Proving this is harder than it sounds, and is essentially equivalent to the fact that $A$ is projective.
In fact, 
Mumford's strategy to show that $A$ and $A^\vee$ are isogenous is to prove that $\lambda_{\mathcal{L}}$ is surjective with finite kernel whenever $\mathcal{L}$ is ample; if one chooses $\mathcal{L}$ defined over $K$, then the resulting isogeny is also defined over $K$.
\end{remark}

One has the following useful fact.
\begin{theorem}
Every abelian variety $A/K$ is $\overline{K}$-isogenous to a principally polarisable abelian variety.
\end{theorem}
\begin{remark}
For a proof in the complex case see Theorem \ref{thm:IsogenyToPPAVOverC}. 
The proof in the general case (see e.g.~\cite[Corollary 1 on page 234]{Mumford70abelianvarieties}) is conceptually the same, but technically more complicated: one needs to rephrase our elementary argument in the language of sections of line bundles, and interpret the quotient $\Lambda'/\Lambda$ appearing in the proof of Theorem \ref{thm:IsogenyToPPAVOverC} as a certain subgroup of $\ker \lambda_{\mathcal{L}}$, where $\mathcal{L}$ is the ample line bundle defining the polarisation. For the sake of completeness, let us point out that the finite subgroup $G \coloneq \Lambda'/\Lambda$ of $A$ we considered is precisely a maximal subgroup of $\ker \left( \lambda_{H} \colon A \to A^\vee\right)$ with the property $\Im H|_{G \times G} \subseteq \Z$: one can make sense of this description for abelian varieties over arbitrary algebraically closed fields, and it is using this description that Mumford proves the theorem in full generality.
\end{remark}

\begin{definition}[Dual homomorphism]
Let $A, B$ be abelian varieties over $K$. Given any $K$-homomorphism $\varphi \colon A \to B$, there is a dual homomorphism $\varphi^\vee \colon B^\vee \to A^\vee$ constructed by applying the universal property of $A^\vee$ with $T=B^\vee$ and $\mathcal{L}=(\varphi \times 1)^*\mathcal{P}_{B \times B^\vee}$ (notice that this is a line bundle on $A \times B^\vee$) to obtain a map $B^\vee=T \to A^\vee$.
Concretely, at the level of points, $\varphi^\vee$ is simply the pullback of line bundles from $B$ to $A$:
\[
\begin{array}{cccc}
\varphi^\vee \colon & B^\vee(K)=\Pic^0(B) & \to & A^\vee(K)=\Pic^0(A) \\
&  [\mathcal{L}] & \mapsto & [\varphi^* \mathcal{L}]
\end{array}
\]
\end{definition}

When $f$ is an isogeny, its dual $f^\vee$ is itself an isogeny, but this is not obvious; for a proof, see for example \cite[Theorem 7.5]{EvGM} or \cite[p. 143]{Mumford70abelianvarieties}. We state this fact formally and give some other properties of $f^\vee$.

\begin{remark}
Cartier duality over fields of positive characteristic can be quite complicated. Over fields of characteristic zero, however, the Cartier dual is easy to describe: if $G$ is a finite commutative group scheme of order $N$,
\[
G^\vee = \Hom(G(\overline{K}), \mu_{N}(\overline{K})).
\]
More precisely: a finite group scheme $G^\vee$ is described by a finite abstract group $H$ together with an action of $\abGal{K}$ on $H$. In this case, the underlying finite group $H$ is $\Hom(G(\overline{K}), \mu_{N}(\overline{K}))$, and the Galois action is the natural one.
\end{remark}
\begin{theorem}[Properties of the dual isogeny]
Let $f \colon A \to B$ be an isogeny. The dual map $f^\vee \colon B^\vee \to A^\vee$ is an isogeny, and its finite kernel $\ker f^\vee$ is the dual of $\ker f$ in the sense of Cartier duality. In particular, $\deg(f^\vee)=\deg(f)$.
\end{theorem}
\begin{example}\label{ex:no-principal-polarisation}
Let $E_1\colon y^2=x^3-x$ and $E_2\colon y^2=x^3-1$ be elliptic curves over $\C$, and let $P_1=(0,0)$ and $P_2=(1,0)$ be $2$-torsion points respectively on $E_1, E_2$. The curves $E_1, E_2$ are not isogenous, since their endomorphism algebras are the distinct imaginary quadratic fields $\Q(i)$ and $\Q(\sqrt{-3})$. Set $A=(E_1\times E_2)/\langle(P_1,P_2)\rangle$, and let $\pi\colon E_1\times E_2\to A$ be the quotient map, of degree $2$ (since it is the quotient map for a subgroup of order $2$). We claim that $A$ does not admit a principal polarisation. If $\lambda : A \to A^\vee$ were a principal polarisation of $A$, its pullback $\mu=\pi^\vee\circ\lambda\circ\pi : E_1 \times E_2 \to E_1 \times E_2$ would have degree $(\deg \pi)(\deg \pi^\vee)=4$. Since $\Hom(E_1,E_2)=\Hom(E_2,E_1)=0$, under the canonical principal polarisations of $E_1$ and $E_2$ the map $\mu$ is of the form $([a],[b])$ for positive integers $a,b$. Moreover, $\mu(P_1,P_2)=0$, so both $a$ and $b$ are even. This contradicts $\deg\mu=a^2b^2=4$. Thus $A$ admits no principal polarisation.
\end{example}

\begin{definition}[Rosati involution]\label{def:Rosati-involution}
Fix a $K$-polarisation $\lambda \colon A \to A^\vee$ of degree $d$, and recall (Remark \ref{rmk:IsogenySymmetric}) that there is an isogeny $\widehat{\lambda} \colon A^\vee \to A$ such that $\widehat{\lambda} \circ \lambda = [d]$. Given an endomorphism $\varphi \colon A \to A$ we define
\[
\varphi^\dagger = \frac{1}{d} \widehat{\lambda} \circ \varphi^\vee \circ \lambda \in \End_K(A) \otimes \Q.
\]
\end{definition}
The equality $(\varphi^\dagger)^{\dagger} =\varphi$ holds, which is why the map $\varphi \mapsto \varphi^\dagger$ is called the Rosati \emph{involution}. Indeed, identifying $A$ with $(A^\vee)^\vee$, the symmetry of a polarisation gives $\lambda^\vee=\lambda$, hence $\widehat{\lambda}^\vee=\widehat{\lambda}$. This implies $(\varphi^\dagger)^\dagger=d^{-2}\widehat{\lambda}\lambda^\vee\varphi\widehat{\lambda}^\vee\lambda=\varphi$. An alternative proof, using the Weil pairing that we now introduce, is given in Corollary \ref{cor:RosatiInvolution}. In the following definition, $A[n]$ denotes the $n$-torsion subgroup of the abelian variety $A$, as in Definition \ref{def:torsion AV 1}.

\begin{definition}[Weil pairing]
Let $A$ be an abelian variety over the field $K$.
For every $n$ not divisible by the characteristic of $K$ there is a canonical pairing
\[
e_n \colon A[n] \times A^\vee[n] \to \G_m[n] = \mu_n,
\]
defined as follows. Let $t \in A[n]$ and $\mathcal{L} \in A^\vee[n]$. By definition, we have $nt=0$ and $\mathcal{L}^{\otimes n} \cong \mathcal{O}$. An application of the theorem of the cube (Theorem \ref{thm:Cube}) gives $[n]^* \mathcal{L} \cong \mathcal{L}^{\otimes n} \cong \mathcal{O}$, so we may fix an isomorphism $u \colon \mathcal{O} \to [n]^* \mathcal{L}$. Denote by $\tau_t \colon A \to A$ the morphism given by translation by $t$.  Pulling back the isomorphism $u \colon \mathcal{O} \to [n]^*\mathcal{L}$ via $\tau_t$ we obtain
\[
\tau_t^*u \colon \tau_t^*\mathcal{O} \to \tau_t^* [n]^* \mathcal{L} \cong ([n] \circ \tau_t)^* \mathcal{L} = [n]^* \mathcal{L}.
\]
Recalling that $\tau_t^* \mathcal{O} = \mathcal{O}$, it follows in particular that $u \circ (\tau_t^* u)^{-1}$ is an automorphism of $[n]^*\mathcal{L}$. Any automorphism of a line bundle on the complete variety $A$ is given by multiplication by a scalar $\zeta \in H^0(A_{\overline{K}},\mathcal{O}^\times)=\overline{K}^\times$, and we define $e_n(t,\mathcal{L})=\zeta$. It is clear from the definition that 
\[
1=e_n(nt,\mathcal{L})=e_n(t,\mathcal{L})^n=\zeta^n
,
\]
so $\zeta$ is in fact an $n$-th root of unity. Any two isomorphisms $u_1,u_2\colon\mathcal{O}\to[n]^*\mathcal{L}$ differ by multiplication by a constant in $\overline{K}^\times$. Replacing $u$ by $cu$ multiplies both $u$ and $\tau_t^*u$ by the same scalar, so $e_n(t,\mathcal{L})$ is independent of the chosen trivialisation. In particular, the Weil pairing is well defined without the need to fix a primitive $n$-th root of unity.
\end{definition}
\begin{remark}Over $\overline{K}$, the map $[n]\colon A\to A$ is a finite \'etale Galois covering. Its group of deck transformations consists precisely of the translations by the points of $A[n]$. Any $n$-torsion line bundle $\mathcal{L}\in A^\vee[n]$ becomes trivial on this covering, and the Weil pairing $e_n(t,\mathcal{L})$ measures the ``difference'' between a chosen trivialisation and its translate by $t$.\end{remark}

\begin{remark}
If $A=E$ is an elliptic curve, the identification $E \cong E^\vee$ given by the canonical principal polarisation allows us to define the Weil pairing directly as a map $E[n] \times E[n] \to \mu_n$. Explicitly, if $O$ is the identity of $E$ and $\mathcal{L}_0=\mathcal{O}_E([O])$, this pairing sends $(P,Q)$ to $e_n(P,\lambda_{\mathcal{L}_0}(Q))$. Here $\lambda_{\mathcal{L}_0}(Q)$ is the class of $\tau_Q^*\mathcal{L}_0\otimes\mathcal{L}_0^{-1}$ in $E^\vee[n]$, as in Definition \ref{def:MumfordsConstruction}. In the general case, if $(A,\mathcal{L})$ is a \emph{polarised} abelian variety, one may still define a Weil pairing on $A[n]$ by the formula
\[
\begin{array}{cccc}
e_n^{\mathcal{L}} \colon & A[n] \times A[n] & \to & \mu_n \\
                    & (t_1, t_2) & \mapsto & e_n\left(t_1,\lambda_{\mathcal{L}}(t_2) \right).
\end{array}
\]
When $\mathcal{L}$ is not a principal polarisation, however, this pairing is in general degenerate.
\end{remark}

\begin{theorem}[Properties of the Weil pairing]\label{thm:PropertiesWeil}
The following hold:
\begin{enumerate}[label={\arabic*.}, wide, labelindent=0pt]
\item We have
\[
e_{mn}(P,Q)^m = e_{n}(mP,mQ) \text{ for } P \in A[mn], Q \in A^\vee[mn],
\]
that is, the Weil pairings on different torsion groups $A[n]$ are compatible.
\item The Weil pairing is perfect: the kernel on both sides is trivial.
\item The Weil pairing is Galois-equivariant: for any $\sigma \in \abGal{K}$ and for any pair of points $P \in A[n], Q \in A^\vee[n]$ we have
\[
e_n(\sigma P, \sigma Q) = \sigma(e_n(P,Q)).
\]

\item For any isogeny $\varphi \colon A \to B$ with kernel contained in $A[n]$ we have
\[
\left\{ y \in A^\vee[n] \colon e_n(x, y)=1 \quad \forall x \in \ker\varphi \right\}
=
\varphi^\vee\bigl(B^\vee[n]\bigr).
\]
\item 
The Weil pairing is compatible with the duality of isogenies, in the sense that if $f \colon A \to B$ is an isogeny, then we have
\[
e_n( f(x), y) = e_n(x, f^\vee(y))
\]
for all $x \in A[n]$ and $y \in B^\vee[n]$.
\item For any polarisation $\mathcal{L}$, the Weil pairing $e_n^{\mathcal{L}}$ introduced above is skew-symmetric on $A[n]$.
\end{enumerate}
\end{theorem}

\subsection{Jacobians}
After the previous general introduction we now turn to more concrete objects; accordingly, we also try to make the exposition more down-to-earth.
We start by considering Jacobians, which are (principally polarised) abelian varieties canonically associated with nice curves. 
Note that we only work with projective curves: the reader should be warned that ``the curve $y^2 = f(x)$, defined over $\Q$" will really mean ``the unique smooth projective curve over $\Q$ birational to the affine curve $\{y^2=f(x)\} \subseteq \A^2_{\Q}$".

\subsubsection{Divisors and their classes}
Assume that $K$ is a \emph{perfect} field and let $C$ be a nice curve over $K$.
\begin{definition}
The \emph{group of divisors} $\operatorname{Div}_C$ is the free abelian group generated by the set $C(\overline{K})$. An element of this group is called a \emph{divisor}, and is nothing but a finite formal linear combination of $\overline{K}$-rational points with integral coefficients. A divisor is \emph{effective} if all its coefficients are non-negative.
We shall represent divisors in the form $D=\sum_{i=1}^k n_i P_i$, where $k \in\N$, $n_i \in \Z$ and $P_i \in C(\overline{K})$ for $i=1,\dotsc,k$.
\end{definition}

So far, this definition only depends on the $\overline{K}$-points of $C$, hence it is not too suitable to study the arithmetic of $C$ over $K$: we need to know what it means for a divisor to be defined over $K$. The definition is straightforward:
\begin{definition}
The group $\abGal{K}$ acts on $C(\overline{K})$ in a natural way, hence it also acts on $\operatorname{Div}_C$. The fixed points for this action form a subgroup, which we denote by $\operatorname{Div}_C(K)$ and whose elements we call \emph{$K$-rational divisors}.
\end{definition}

\begin{remark}
If all the points $P_1,\dotsc,P_r$ are rational, a divisor $D=\sum n_i P_i$ is certainly $K$-rational, because $\abGal{K}$ acts trivially on the $P_i$. However, the divisor $D=\sum n_i P_i$ can be $K$-rational even if the points $P_i$ are not: consider for example the curve $C/\Q \colon x^2+y^2=-1$. The divisor $(0,i) + (0,-i)$ is $\Q$-rational, but the points $(0,i)$ and $(0,-i)$ are not defined over $\Q$.
\end{remark}

We also remark that there is an obvious numerical invariant attached to a divisor, namely its degree:
\begin{definition}[Degree of a divisor]
The \emph{degree} of a divisor $D = \sum_{i=1}^k n_i P_i$ is
\[
\deg D = \sum_{i=1}^k n_i \in \Z.
\]
We interpret $\deg$ as a group homomorphism $\Div_C \to \Z$; its kernel will be denoted by $\Div^0_C$ (the subgroup of divisors of degree 0).
\end{definition}

%

\begin{definition}[Divisor of a function]
Let $f \in \overline{K}(C) \setminus \{0\}$ be a rational function on $C$. The divisor of $f$ is
\[
\operatorname{div}(f) = \sum_{P} v_P(f) \, P,
\]
where the sum is over all the points of $C(\overline{K})$. Here $v_P(f)$ is the order of vanishing of $f$ at $P$; if $f$ has a pole of order $k$ at $P$, then $v_P(f)=-k$.
\end{definition}
\begin{remark}
Any nonzero rational function has only finitely many zeroes and poles, hence the sum defining $\operatorname{div}(f)$ is finite and $\operatorname{div}(f)$ is indeed a divisor.
\end{remark}
\begin{definition}[Principal divisor]
A divisor is said to be \emph{principal} if it is of the form $\operatorname{div}(f)$ for some nonzero rational function $f \in \overline{K}(C)$. One can check that a principal divisor has degree 0 (a rational function on a smooth projective curve has as many zeroes as poles). We denote by $\operatorname{Princ}_C \subseteq \Div^0_C$ the subgroup of principal divisors.
\end{definition}
\begin{definition}[Picard group]\label{def: Picard group curve}
We define 
\[
\operatorname{Pic}_C \coloneq \Div_C/\operatorname{Princ}_C
\]
and
\[
\operatorname{Pic}^0_C \coloneq \Div^0_C/\operatorname{Princ}_C;
\]
equivalently, we observe that $\deg \colon \Div_C \to \Z$ descends to $\deg \colon \operatorname{Pic}_C \to \Z$, and we define $\operatorname{Pic}^0_C$ to be the kernel of the latter homomorphism.
\end{definition}

\begin{definition}
Two divisors $D_1, D_2$ are said to be \emph{linearly equivalent} (written as $D_1 \sim D_2$) if they differ by a principal divisor. We denote by $[D]$ the class of the divisor $D$ in $\operatorname{Pic}_C$.
\end{definition}

\begin{notation}
We shall also write $\Princ_C(\overline{K})$ or $\Div_C(\overline{K})$ when we want to stress that we are considering these as abstract groups ($\Div_C, \Princ_C$ have in principle a richer structure).
\end{notation}

By analogy with the definition of $\Div_C(K)$, we now define the group $\operatorname{Princ}_C(K)$ as the $\abGal{K}$-invariant subgroup of $\operatorname{Princ}_C(\overline{K})$. Notice that $\operatorname{Princ}_C(\overline{K})$ is a Galois submodule of $\operatorname{Div}^0_C(\overline{K})$, which is in turn a Galois submodule of $\operatorname{Div}_C(\overline{K})$. This allows us to take quotients \emph{in the category of Galois modules}, and finally leads to the definition of the $K$-points of $\operatorname{Pic}^0_C$.
\begin{definition}
We let $\displaystyle \Pic_C(K)$ be the $\abGal{K}$-invariant subgroup of $\frac{\Div_C(\overline{K})}{\Princ_C(\overline{K})}$, and denote by $\displaystyle \Pic^0_C(K)$ its subgroup given by the $\abGal{K}$-invariants of $\frac{\Div^0_C(\overline{K})}{\Princ_C(\overline{K})}$.
\end{definition}

\begin{theorem}[Jacobian of a curve]
Let $C$ be a nice curve over $K$. There is an abelian variety $J$ over $K$
such that there is an isomorphism of $\abGal{K}$-modules $\Pic^0_C \cong J(\overline{K})$. This abelian variety $J$ is called the Jacobian variety, or just the Jacobian, of $C$. The dimension of $J$ is the genus of $C$ (in particular, when $g(C)=0$ the Jacobian is the trivial $0$-dimensional abelian variety).
\end{theorem}
\begin{remark}
Jacobian varieties are principally polarised (even when $C(K)=\emptyset$): when $C(K) \neq \emptyset$ and $g \geq 2$, the (linear equivalence class of a) divisor giving the principal polarisation may be obtained as the image of any map of the form
\[
\begin{array}{ccc}
\operatorname{Sym}^{g-1}(C) & \to & \Jac(C) \\
(P_1,\dotsc,P_{g-1}) & \mapsto & P_1+\dotsc+P_{g-1} - (g-1)O,
\end{array}
\]
where $O$ is a point in $C(K)$. 

Suppose that $C(K)$ is nonempty and fix a point $P \in C(K)$. We can associate with this point an embedding of $C$ in $\Jac(C)$: identifying $\Jac(C) \cong \Pic_C^0$, we may define a map $C \to \Pic_C^0 \cong \Jac(C)$ by the formula $Q \mapsto [Q-P]$. This is called the \emph{Abel-Jacobi map}; it is an embedding provided that $C$ has positive genus. We will see below in Example \ref{ex:JacEC} that, when $C=E$ is an elliptic curve and $P=O$ is the origin of the group law, the map $E \to \Jac(E)$ defined above is an isomorphism.
\end{remark}

As is the case for many important objects in algebraic geometry, Jacobians also satisfy a useful universal property:

\begin{proposition}[Universal property of the Jacobian]
Let $C/K$ be a nice curve of positive genus with Jacobian $J/K$. Fix a rational point $P \in C(K)$ and denote by $i\colon C \to J$ the corresponding embedding of $C$ into its Jacobian.
The Jacobian $J$ satisfies the following universal property: for any abelian variety $A$ and for any algebraic morphism $f \colon C \to A$ such that $f(P)=0_A$ there is a unique homomorphism of abelian varieties $g \colon J \to A$ such that the following diagram commutes:
\[
\begin{tikzcd}
C \arrow[dr, "f"'] \arrow[r, "i"] & J \arrow[d, "g"] \\
& A
\end{tikzcd}
\]
\end{proposition}

\begin{remark}
To be more precise, this is the universal property of the Albanese variety of the pair $(C, P)$ (that is: the Albanese variety is by definition the initial object with respect to maps from $(C,P)$ to abelian varieties that carry $P$ to the neutral element). The point is that -- assuming for simplicity that $K$ is perfect -- the Albanese variety of $(C,P)$ is naturally isomorphic to the \emph{dual} of $\Pic_C^0$. Finally, one obtains $\operatorname{Alb}(C,P) \cong (\Pic_C^0)^\vee \cong \Pic_C^0$, because Jacobians are canonically principally polarised.

To construct the isomorphism $\operatorname{Alb}(C,P) \cong (\operatorname{Pic}^0_C)^\vee$, notice that given a map from $C$ to an abelian variety $A$ carrying $P$ to $0_A$ we obtain by pullback a map
\[
\operatorname{Pic}_A \to \operatorname{Pic}_C,
\]
which maps the connected component of the identity of the former into the connected component of the identity of the latter, whence a map
\[
\operatorname{Pic}^0_A \to \operatorname{Pic}^0_C,
\]
and finally, by duality, a homomorphism
\[
(\operatorname{Pic}^0_C)^\vee \to (\operatorname{Pic}^0_A)^\vee.
\]
Now our discussion of dual abelian varieties (Section \ref{sect:DualAV}) implies that $(\operatorname{Pic}^0_A)^\vee \cong (A^{\vee})^{\vee} \cong A$, so from a map of pointed varieties $(C,P) \to (A,0_A)$ we have constructed a homomorphism $(\operatorname{Pic}^0_C)^\vee \to (\operatorname{Pic}^0_A)^\vee \cong A$. One can then show that the composition
\[
C \to (\operatorname{Pic}^0_C)^\vee \to A
\]
is the map we started with.
\end{remark}

\begin{example}[An elliptic curve is its own Jacobian]\label{ex:JacEC}
Let $E/K$ be an elliptic curve. In this example, we show that $E$ is isomorphic to its Jacobian variety by using the classical construction of the group law on an elliptic curve. For simplicity, we work over a field of characteristic different from $2$ and $3$, so that any elliptic curve over $K$ can be represented by a short Weierstrass model.

We claim that any divisor of degree $0$ on $E$ is linearly equivalent to a divisor of the form $P-O$, where $O$ is the origin of the group law on $E$ and $P$ is a point on $E$. To see this, embed $E$ as a plane cubic in $\bbP^2$ (in short Weierstrass form, with the origin of the group law being the point $[0 \colon 1 \colon 0]$). Any line meets $E$ at three points $P_1, P_2, P_3$; it is then easy to see that the divisor $P_1+P_2+P_3-3O$ is principal. Now fix two points $P, Q$ lying in the affine part of $E$ (not both on the same vertical line). The line through $P$ and $Q$ meets $E$ exactly at a third point $R$ (possibly coincident with $P$ or $Q$), so the divisor of the rational function corresponding to this line is $P+Q+R -3O$. It follows that $P+Q \sim 3O-R$. On the other hand, if $R$ and $R'$ lie in the finite plane on the same vertical line, then the line in question is $x-x(R)=0$, and the function $x-x(R)$ has zeroes at $R, R'$ and a double pole at $O$. It follows that $R+R' \sim 2O$. Combined with $P+Q \sim 3O-R$, this yields $P+Q \sim 3O-R \sim 3O-(2O-R') = O+R'$. On the one hand, this recovers the classical construction of the addition law on elliptic curves; on the other, it gives us an algorithm to transform any divisor of the form $\sum_{i=1}^k P_i$ into one of the form $(k-1)O + Q$. Now consider a general divisor of degree $0$, $D=\sum_{i=1}^k P_i - \sum_{j=1}^k Q_j$. We already know how to construct points $Q_j'$ such that $Q_j+Q_j' \sim 2O$, so $D$ is linearly equivalent to $\sum_{i=1}^k P_i+ \sum_{j=1}^k (Q'_j-2O)$; applying our reduction algorithm to $\sum_{i=1}^k P_i+ \sum_{j=1}^k Q'_j$, we find that it is linearly equivalent to a divisor of the form $R+(2k-1)O$, where $R$ is a single point on $E$. Putting everything together, we obtain as desired
\[
D \sim R + (2k-1)O - 2kO = R-O.
\]
It follows in particular that the map $E \to \Pic^0(E)$ given by $P \mapsto [P-O]$ is surjective, and it's not hard to see that it is injective (if $P-O$ were a principal divisor, $P-O=\operatorname{div}(f)$, then $f:E \to \bbP^1$ would be a map of degree 1, hence an isomorphism, which is impossible since $E$ has genus 1 while $\bbP^1$ has genus 0).
\end{example}

\begin{remark}
It is often useful to think of the Jacobian of a curve $C$ as \emph{the abelian variety whose regular differentials are the same as those of $C$}. More precisely, let $i \colon C \hookrightarrow J$ be the embedding induced by the fixed rational point on $C$. Then there is a canonical pullback map
\[
i^* \colon H^0(J,\Omega^1_J) \to H^0(C, \Omega^1_C),
\]
which can be shown to be an isomorphism.
\end{remark}

\subsubsection{Example: adding points on a Jacobian}
We take the curve of this example from \cite{MR1406090}.
Consider the curve 
\[
C  / \Q \colon y^2 = x(x-1)(x-2)(x-5)(x-6).
\]
It has a single point at infinity, which we denote by $\infty$; we fix the embedding $\iota$ of $C$ into $\Jac(C)$ given by mapping $\infty$ to $[0] \in \Jac(C)(\Q)$.
We consider the degree-0 divisors on $C$
\[
A = (0, 0) + (1, 0) -2 \infty \quad \text{ and } \quad B = (2, 0) + (3, 6)-2\infty.
\]
Their divisor classes $[A]$ and $[B]$ are points of $\Jac(C)$. We show some manipulations involving the points $[A]$ and $[B]$ on the Jacobian; in particular, we want to determine a divisor $D_1+D_2-2\infty$ on $C$ that represents the same divisor class as $[A+B]$.

We first find a function vanishing at all the points in the support of $A$ and $B$: for example,
$f \coloneq y - x(x - 1)(x - 2)$ works. Substituting this into the equation for $C$ we find $x(x - 1)(x - 2)(x - 3)(x^2 - x + 10) = 0$. Thus the divisor of the function $f$ is
\[
(0,0) + (1,0) + (2,0) + (3,6) + C_1 + C_2 - 6\infty=A+B+C_1+C_2-2\infty,
\]
where
\[
C_1 = \Bigl( \frac{1}{2} + \frac{1}{2}\sqrt{-39}, 15-5\sqrt{-39} \Bigr), \quad C_2 =  \Bigl( \frac{1}{2} - \frac{1}{2}\sqrt{-39}, 15+5\sqrt{-39} \Bigr).
\]
Since $[\operatorname{div} f]=[0]$ in $\Jac(C)$, we have thus proven
\[
[A+B] = [A+B-\operatorname{div} f] = [2\infty-C_1-C_2].
\]
Suppose we want to find a representative of the form $[D_1+D_2-2\infty]$: it is then enough to find a function with divisor $C_1+C_2+D_1+D_2-4\infty$. But we already know such a function! Indeed, $C_1, C_2$ are zeroes of $x^2-x+10$. This function is of degree 4 and is regular on the affine chart of our curve, so its divisor is
\[
\operatorname{div}(x^2-x+10) = C_1+C_2 + D_1+D_2-4\infty
\]
with
\[
D_1 = \Bigl( \frac{1}{2} + \frac{1}{2}\sqrt{-39}, -15+5\sqrt{-39} \Bigr), \quad D_2 =  \Bigl( \frac{1}{2} - \frac{1}{2}\sqrt{-39}, -15-5\sqrt{-39} \Bigr).
\]
Hence $[A+B] = [D_1+D_2-2\infty]$ in $\Jac(C)$; notice that once we find $[A+B]=[2\infty-C_1-C_2]$ we also obtain the same conclusion by observing that the hyperelliptic involution induces $-\operatorname{id}$ on the Jacobian, which means that (denoting by $\overline{P}$ the point $(x,-y)$ when $P=(x,y)$) we have
\[
\iota(\overline{P}) = -\iota(P) \Leftrightarrow [\overline{P}-\infty] = [\infty-P].
\]
To see that the hyperelliptic involution induces $-\operatorname{id}$, observe that, for any affine point $P=(a,b)$, the function $x-a$ has divisor $P+\overline{P}-2\infty$, with the two zeroes coinciding when $b=0$. In particular, $[P-\infty ]+[\overline{P} - \infty]=0$, that is, $\iota(P) = -\iota(\overline{P})$.
Applying this in our particular example we obtain
\[
[A+B] = [2\infty-C_1-C_2] = [-2\infty+\overline{C_1}+\overline{C_2}] = [D_1+D_2-2\infty].
\]

\section{Abelian varieties: arithmetic}

\subsection{Abelian schemes over arbitrary bases}
For arithmetic applications it is extremely useful to have a notion of abelian variety also over arbitrary bases -- that is, we want to treat the general situation of an abelian variety defined over an arbitrary base scheme $S$ and not just over the spectrum of a field (in this case, one speaks more precisely of \emph{abelian schemes}). Here is the general definition:
\begin{definition}[Abelian schemes over general bases] Let $S$ be a scheme. A $g$-dimensional abelian scheme over $S$ is a group scheme $\mathcal{A} \to S$ such that the structure morphism
$\mathcal{A} \to S$ is of finite presentation, proper, smooth, with all fibres geometrically connected and of dimension $g$.
\end{definition}

\begin{remark}
If $S$ is noetherian (which is the case in most arithmetic applications), then \emph{finite presentation} can be replaced by \emph{finite type}. Recall that $\pi \colon \mathcal{A} \to S$ is of finite type if the following holds: there is a cover of $S$ by open affine subschemes $S_i=\Spec(R_i)$ and a finite cover of every $\pi^{-1}(S_i)$ by open affine subschemes $\mathcal{A}_{ij} = \Spec(B_{ij})$, such that $B_{ij}$ is a finitely generated $R_i$-algebra for every $i,j$.
 \end{remark}

This somewhat abstract definition essentially amounts to asking for a family of abelian varieties $A_s$ that varies algebraically with respect to the (geometric) point $s \in S$. It is important in the arithmetic setting because it provides a geometric framework for the notion of \emph{reduction modulo $p$}, which can be defined in concrete terms (e.g.~using equations) for elliptic curves, but is much harder to describe in such elementary terms for higher dimensional abelian varieties. Of course, abelian schemes are also very useful for studying the more geometric problem of understanding families of abelian varieties that depend algebraically on some parameters, a typical example being the Jacobian scheme of a family of curves.

\subsection{Models of algebraic groups}
Let $K$ be a number field (respectively, a local field) and let $X$ be a variety over $K$. In many arithmetic questions it is important to control the behaviour of $X$ when we reduce modulo primes (respectively, the unique prime) of $K$. 
This requires the notion of a \emph{model} of $X$ over the ring of integers $\OK$, or at least over a non-empty open subscheme of $\Spec \OK$ containing the generic point (in the local case, the only such open subscheme is $\Spec \OK$ itself). We start by discussing this notion for general varieties before specialising to abelian varieties.

\begin{definition}
Let $S$ be a non-empty open subscheme of $\Spec \OK$ containing the generic point $\Spec K$. 
A \emph{model} of $X$ over $S$ is a flat, finitely presented $S$-scheme $\mathcal{X}$ together with an isomorphism
\[
\mathcal{X} \times_{S} \Spec K \;\cong\; X.
\]
\end{definition}

Intuitively, a model provides a way to interpolate $X$ ``at all but finitely many places'' of $\OK$, so that one can make sense of its reduction modulo primes. Specifically, if $v$ is a closed point of $S$ with residue field $\kappa(v)=\F_v$, the reduction of $\mathcal{X}$ modulo $v$ is the base-change $\mathcal{X} \times_S \Spec(\kappa(v))$.

Note that any open subscheme $S$ of $\Spec \OK$ contains $\Spec \OK[1/N]$ for some positive integer $N$. The scheme $\Spec \OK[1/N]$ is obtained from $\Spec \OK$ by deleting the closed points corresponding to the primes of $\OK$ with residue characteristic dividing $N$.
In particular, a model over $S \coloneq  \Spec \OK[1/N]$ gives us reduction maps
\[
\mathcal{X}(\OK[1/N]) \;\longrightarrow\; \mathcal{X}(\F_v)
\]
for all finite places $v$ such that $\operatorname{char}(\F_v) \nmid N$. 

\begin{example}\phantom{~}
\begin{enumerate}[label={\arabic*.}, wide, labelindent=0pt]
  \item The projective space $\bbP^n_K$ has a natural model over $\OK$, namely $\bbP^n_{\OK}$. Reduction modulo a prime ideal $v$ gives the projective space $\bbP^n_{\F_v}$.
  \item If $E/K$ is an elliptic curve given by a Weierstrass equation
  \[
  y^2 + a_1 xy + a_3 y = x^3 + a_2 x^2 + a_4 x + a_6, \quad a_i \in K,
  \]
  then, after scaling the coefficients to lie in $\OK$, we obtain a model $\mathcal{E}/\OK$. This model is not unique: in fact one often chooses a ``minimal model'' at each place $v$, meaning that the discriminant has minimal valuation at $v$ \cite[Chap.~VII]{SilvermanAEC}.
  \item More generally, any smooth projective curve $C/K$ admits a regular proper model $\mathcal{C}$ over $\OK$, although constructing such a model is delicate. 
  \item For algebraic groups such as $\GL_n$ or $\SL_n$, one can define natural group schemes $\GL_{n,\OK}$, $\SL_{n,\OK}$ over $\OK$, which serve as models of $\GL_{n,K}$ and $\SL_{n,K}$.
\end{enumerate}
\end{example}

Reduction maps are especially useful if $X$ and $\mathcal{X}$ are projective: indeed, in this case (since $\OK[1/N]$ is a normal domain) we have $X(K)=\mathcal{X}(K)=\mathcal{X}(\mathcal{O}_K[1/N])$, hence we have a reduction map $X(K) \to \mathcal{X}(\F_v)$ for all but finitely many primes $v$. This is justified in detail in the following proposition, which is a special case of \cite[Exercise 15.3]{GoertzWedhornI}; given the lack of a detailed reference, we provide a short proof below. 

Both in this argument and in the next section, we rely on standard \emph{spreading-out} results, for which we refer the reader to the wonderful exposition of \cite[\S3.2]{PoonenRatPts}. Roughly speaking, \emph{spreading out} refers to the following idea: given a scheme $X$ over a field $K$, satisfying some property $\mathcal{P}$ (such as being smooth over $\operatorname{Spec} K$, or separated, ...), for many interesting $\mathcal{P}$ one can find a dense open subscheme $U$ of $\operatorname{Spec} \mathcal{O}_K$ and a $U$-scheme $\mathcal{X} \to U$ that again satisfies $\mathcal{P}$ and whose generic fibre is $X$ (that is, $\mathcal{X} \times_{U} \operatorname{Spec} K \cong X$) -- that is, we have extended (or \emph{spread out}) a property valid at the generic point of $\operatorname{Spec} \mathcal{O}_K$ to a whole dense open subscheme.

\begin{proposition}\label{prop: XRXK}
Let $K$ be a number field with ring of integers $\mathcal{O}_K$, let $N\ge 1$, and set $R=\mathcal{O}_K[1/N]$. Let $\mathcal{X}$ be a proper $R$-scheme and write $X\coloneq \mathcal{X}\times_R K$. The natural map
\[
\Phi \colon \mathcal{X}(R)\longrightarrow \mathcal{X}(K)
\]
is a bijection.
\end{proposition}

\begin{proof}
Let $S=\Spec R$ and $\eta$ be the generic point $\Spec K$ of $S$.

We first show injectivity. By definition, properness implies that $\mathcal{X}\to S$ is separated. This implies that the subscheme of $S$ where two sections $\alpha_1, \alpha_2 \colon \Spec R \to \mathcal{X}$ agree is closed. If $\alpha_1, \alpha_2$ agree on $\eta$ (that is, if $\Phi(\alpha_1)=\Phi(\alpha_2)$), then -- since the generic point is dense in $S$ -- this closed set must be all of $S$, hence $\alpha_1=\alpha_2$ as desired.

Now we prove surjectivity. Let $\alpha\colon \eta\to \mathcal{X}$ be a $K$-point. Since $\mathcal{X}\to S$ is of finite type, by spreading out (see e.g.~\cite[Theorem 3.2.1(iii)]{PoonenRatPts}) $\alpha$ extends to a morphism $\beta \colon U \to \mathcal{X}$ on some nonempty open $U\subset S$. The complement $S\setminus U$ is finite, consisting of closed points $p_1,\dotsc,p_n$. For each $p_i$, the valuative criterion for properness provides a unique extension $\Spec \mathcal{O}_{S,p_i}\to\mathcal{X}$. By spreading out again, these morphisms further extend to an open subscheme of $S$. By separatedness, these local extensions glue with $\beta$ on overlaps (because they all agree with $\beta$ on $\eta$). Iterating over all missing points produces an $R$-section extending $\alpha$.  
\end{proof}

We make this abstract discussion more concrete with the following example.
\begin{example}
Consider the elliptic curve
\[
E/\Q\colon Y^2Z = X^3 - 2XZ^2 \subset \bbP^2_{\Q},
\]
and take the obvious model
\[
\mathcal{E}/\Z[1/6]\colon Y^2Z = X^3 - 2XZ^2 \subset \bbP^2_{\Z[1/6]}.
\]
By the above discussion, we have a reduction map
\[
\mathrm{red}_5\colon \mathcal{E}(\Z[1/6])=\mathcal{E}(\Q)\longrightarrow
\mathcal{E}(\F_5)=E(\F_5);
\]
we now compute it explicitly on an example. The point $P=\bigl(\tfrac{9}{4},-\tfrac{21}{8}\bigr)\in E(\Q)$ in the affine patch $y^2 = x^3-2x$ of $E/\Q$ corresponds to the projective point $[18 \colon -21 \colon 8] \in E(\Q) \subset \bbP^2(\Q)$. In fact, this point makes perfect sense also in $\mathcal{E}(\Z[1/6])$, and we 
may reduce its coordinates modulo $5$:
\[
[18\colon -21\colon 8] \equiv [3 \colon -1 \colon 3] \pmod 5.
\]
If we write it again in affine coordinates, this gives the point $(1, 3) \in E(\F_5)$. A quick check on the reduced equation $y^2=x^3-2x$ over $\F_5$ shows $3^2\equiv 1^3-2\cdot 1\equiv 4\pmod{5}$, so $(1,3)$ lies in $E(\F_5)$ as claimed.
\end{example}

\subsection{Models of abelian varieties}

For abelian varieties $A/K$, one can find models over open subsets of $\Spec \mathcal{O}_K$ not just for the variety, but for the full group structure. That is, there exists a positive integer $N$ and a model $\mathcal{A} \to S \coloneq  \Spec \OK[1/N]$ that is also an abelian scheme. More explicitly, $\mathcal{A} \to S$ is a smooth proper scheme of finite presentation, and there are an addition map $+\colon\mathcal{A}\times_{\mathcal{O}_K[1/N]}\mathcal{A}\to \mathcal{A}$, an inverse map $[-1]\colon\mathcal{A}\to \mathcal{A}$, and a zero section $0\colon\Spec \mathcal{O}_K[1/N]\to \mathcal{A}$ that are morphisms of schemes over $\mathcal{O}_K[1/N]$ and satisfy the relations given in Section \ref{sect: group schemes}. The existence of such a model follows from standard spreading-out arguments combined with the fact that abelian varieties are proper and smooth over $K$. We sketch the proof.

\begin{theorem}[Finiteness of the set of places of bad reduction]\label{thm:FinitenessBadReduction}
Let $A/K$ be an abelian variety over a number field. There exists an open subscheme $S$ of $\Spec \OK$ and a smooth proper group scheme $\mathcal{A} \to S$ with $\mathcal{A} \times_S \Spec K \cong A$.
\end{theorem}
\begin{proof}
Since $A \to \Spec(K)$ is smooth and projective over the generic point of $\Spec(\OK)$, by spreading out, $A$ can be extended to a smooth projective variety (not necessarily an abelian scheme) over some dense open subset of $\Spec(\OK)$.
The same holds for the multiplication, inverse, and unit maps; since the intersection of these finitely many open subsets is open (and non-empty, because it contains the generic point), this implies that we can find a common open subscheme $S$ of $\Spec(\OK)$ over which $A$, $m$, $i$, and the unit section all extend. Given that these maps satisfy the commutative diagrams that encode the property of being a group scheme \emph{generically}, they also satisfy the same properties over $S$: this means precisely that there is an abelian scheme $\mathcal{A} \to S$ with generic fibre $A$. 
\end{proof}

\begin{remark}
In layman's terms, the above argument is essentially the following:  choose any set of equations for $A$ and for the morphisms that give $A$ its group structure. Let $T$ be the set of primes of $\OK$ that divide the denominator of at least one coefficient of at least one equation of the given presentation (there are only finitely many of them). On $\Spec \OK \setminus T$, the given equations give a model for $A$ and its operations. Since this scheme is generically smooth, by restricting to a further open subscheme of $\Spec \OK$ if necessary, we can also ensure that $\mathcal{A} \to S$ is smooth.
\end{remark}

Once we have such a smooth projective model, we can define reduction maps for all primes $v$ of $\mathcal{O}_K$ not dividing $N$:
\[
A(K) = \mathcal{A}(\mathcal{O}_K[1/N]) \xrightarrow{\operatorname{red}_v} \mathcal{A}(\F_v),
\]
where the first identification uses Proposition \ref{prop: XRXK} and the fact that $\mathcal{A}$ is proper. Moreover, since $\Spec \F_v \to S$ is a morphism of $S$-schemes, the general formalism of group schemes (see Remark \ref{rmk: group schemes are functors}) implies that the reduction map is a group homomorphism: we have
\[
\operatorname{red}_v(x+y) = \operatorname{red}_v(x) + \operatorname{red}_v(y) \quad \forall x,y \in A(K), \forall v \nmid N.
\]
We will often write simply $x \bmod v$ for the image of $x$ under the reduction map $\operatorname{red}_v$.

\subsubsection{Good and bad reduction}

The above discussion leads naturally to the notion of places of good and bad reduction for an abelian variety.

\begin{definition}\label{def:good and bad reduction global}
Let $A/K$ be an abelian variety.
A finite place $v$ of $K$ is said to be of \emph{good reduction} for $A$ if there exists an open subscheme $S \subset \Spec \mathcal{O}_K$ containing $v$ and an abelian scheme $\mathcal{A} \to S$ whose generic fibre is $A$. In other words, $A$ has good reduction at $v$ if it can be extended to a smooth proper model with group structure over some open subset of $\Spec \mathcal{O}_K$ that includes $v$. If no such model exists for $v$, we say that $A$ has \emph{bad reduction} at $v$. 
\end{definition}

\begin{notation}
Let $A/K$ be an abelian variety over a number field, and let $v$ be a place of good reduction for $A$. We shall write $A(\F_v)$ to mean $\mathcal{A}(\F_v)$, where $\mathcal{A} \to \Spec(\Oring_{K_v})$ is an abelian scheme extending $A_{K_v}$.
\end{notation}

Note that Theorem \ref{thm:FinitenessBadReduction} implies that there are only finitely many primes of bad reduction: the existence of a smooth proper model $\mathcal{A} \to \Spec \OK[1/N]$ for some $N \geq 1$ guarantees that all primes not dividing $N$ are of good reduction.

\begin{remark}\label{rmk: good bad reduction are local}
Good and bad reduction are local properties: their definitions could (and in fact \emph{should}) be given in the context of abelian varieties over local fields. In this setting, given an abelian variety $A \to \Spec(K)$, where $K= \operatorname{Frac}(R)$ is the field of fractions of a discrete valuation ring $R$, we say that $A$ has good reduction if there exists an abelian scheme $\mathcal{A} \to \Spec(R)$ whose generic fibre is $A$. It is a not-entirely-straightforward exercise to show that an abelian variety $A$ over the number field $K$ has good reduction at the finite place $v$ (in the sense of Definition \ref{def:good and bad reduction global}) if and only if $A \times_K K_v$ has good reduction over the local field $K_v$. We will indifferently make use of either description, depending on context. To simplify the discussion, we discuss global models only over number fields and leave the straightforward adaptations to the local case to the reader. Conversely, we work over local fields when considering field extensions and Jacobians, and leave it to the reader to formulate the analogous statements over global fields.
\end{remark}

\begin{remark}\label{rmk: Fontaine Abrashkin}
Let $A$ be an abelian variety over a number field $K$. It is not true in general that $A$ extends to an abelian scheme over $\OK$ (in fact, this \emph{almost never} happens). In particular, it is a famous theorem due to Fontaine \cite{MR807070} and Abrashkin that there are no (nonzero) abelian schemes over $\Spec(\Z)$. On the other hand, there are \emph{some} abelian varieties that extend to the full ring of integers of a number field: one of the most famous examples is the elliptic curve $E$ over $K=\Q(\sqrt{29})$ with equation
\[
y^2 + xy + a^2y = x^3,
\] where $a = \frac{5+\sqrt{29}}{2} \in \OK^\times$. This elliptic curve extends to an abelian scheme over all of $\OK$, or equivalently, it has good reduction at all the primes of $\OK$. Indeed, the displayed equation has discriminant $\Delta=-a^{10}\in\mathcal{O}_K^\times$. Note that an integral short Weierstrass equation $y^2=x^3+Ax+B$ has discriminant $-16(4A^3+27B^2)$, which is divisible by $16$ and hence cannot be a unit at primes above $2$. Thus in order to find elliptic curves with good reduction everywhere one must allow general Weierstrass equations, as in this example.
\end{remark}

\subsubsection{The case of elliptic curves}

For elliptic curves, the notions of good and bad reduction can be described very concretely. Let $E/K$ be an elliptic curve defined over a number field $K$, given by a Weierstrass equation with coefficients in $K$. After multiplying through by a suitable common denominator, one can find an integral Weierstrass equation with coefficients in $\mathcal{O}_K$. Not every such integral equation gives the ``best" model at a given prime, but there is always a \emph{minimal Weierstrass model} for each finite place $v$ of $K$: this is a Weierstrass equation defined over the local ring $\mathcal{O}_{K_v}$ whose discriminant has minimal possible valuation among all integral models at $v$ (see \cite[\S VII.1 and \S VII.2]{SilvermanAEC}).

The discriminant of a minimal Weierstrass model detects the reduction type. Let $\Delta_v$ denote the discriminant of a minimal Weierstrass model at $v$: then $E$ has good reduction at $v$ if and only if $v(\Delta_v)=0$, i.e., the discriminant is a unit in $\mathcal{O}_{K_v}$. Equivalently, when one reduces the coefficients of the minimal model modulo the maximal ideal of $\mathcal{O}_{K_v}$, the resulting cubic curve in $\bbP^2_{\F_v}$ is smooth, and hence defines an elliptic curve over the residue field $\F_v$. If instead $v(\Delta_v) > 0$, then (by minimality of $\Delta_v$) no model of $E$ can be smooth at $v$, hence $E$ does \emph{not} extend to an abelian scheme over an open of $\Spec \OK$ that contains $v$. Thus, $E$ has bad reduction at $v$. 

\begin{example}
As a concrete example, consider again the elliptic curve
\[
E\colon y^2 = x^3 - 2x
\]
over $\Q$. A straightforward computation shows that the discriminant of this model is $\Delta = -16(4(-2)^3+27\cdot 0^2) = 2^9$. The only prime divisor of the discriminant is $2$. Thus, this equation gives a smooth model over $\Z[1/2]$, and $E$ has good reduction at least at all odd primes. A less straightforward computation shows that this equation is in fact also minimal at $2$, hence $E$ has bad reduction at $2$. 

On the other hand, consider the elliptic curve $E' \colon y^2=x^3 + 5^6$ over $\Q$. The given equation gives a model that is smooth away from $2, 3$ and $5$; however, $E'$ has in fact good reduction at $5$, because $E'$ is isomorphic over $\Q$ to the elliptic curve with equation $v^2 = u^3+1$ (to see this, set $y=5^3v, x=5^2u$), and the discriminant of this new model is $-432=-2^4 3^3 \not \equiv 0 \pmod 5$. The valuations of this discriminant at $2$ and $3$ are $4$ and $3$, respectively. Since these are less than $12$, the equation is minimal at both primes, so $E'$ has bad reduction at $2$ and $3$.
\end{example}




\subsubsection{Semistable reduction}
Let $A$ be an abelian variety over a number field $K$. The discussion of good and bad reduction above is formulated in terms of models over an unspecified open subscheme of $\Spec \OK$, and -- as we observed in Remark \ref{rmk: Fontaine Abrashkin} -- there is in general no abelian scheme over $\Spec \OK$ that extends $A/K$, so it is not clear that there should exist a preferred model $\mathcal{A} \to \Spec \OK$.
It came as quite a surprise when Néron observed that abelian varieties do admit a ``best possible'' model over the whole ring of integers; the concrete meaning of \emph{best possible} is quite subtle, and is captured by the following definition.

\begin{definition}[Néron model]
Let $A/K$ be an abelian variety over a number field or a local field $K$. The \emph{Néron model} of $A$ is a smooth separated \emph{but in general not proper} group scheme $\mathcal{A} \to \Spec(\OK)$ with the following property, called the \emph{Néron mapping property}:
if $\mathcal{X}$ is a smooth separated scheme over $\OK$, then any $K$-morphism between generic fibres $\mathcal{X}_K \to \mathcal{A}_K$ can be extended to a unique $\OK$-morphism from $\mathcal{X}$ to $\mathcal{A}$.
\end{definition}

It is a deep theorem, due in its original form to Néron \cite{Ner64} and in full generality to Raynaud \cite{RaynaudNeron}, that (semi)abelian varieties over the fraction field of a Dedekind domain admit a Néron model. As is usually the case for objects defined in terms of a universal property, Néron models are unique up to unique isomorphism.

\medskip

Néron models allow us to define the notion of semistable reduction. This is a technical notion, but a very useful one.
\begin{definition}[Semistable reduction]\label{def:semistable reduction}
Let $A$ be an abelian variety over the number field $K$, with Néron model $\mathcal{A}/\OK$. We say that $A$ has \emph{semistable reduction} at the finite place $v$ if the connected component of the identity of $\mathcal{A} \times_{\OK} \F_v$ is an extension of an abelian variety by an algebraic torus. We say that $A$ has semistable reduction if this holds for every finite place $v$ of $K$.
\end{definition}

\begin{remark}
More explicitly, $A$ has semistable reduction at $v$ if the special fibre at $v$ of $\mathcal{A}$ fits into an exact sequence
\[
0 \to T \to \mathcal{A}_v^0 \to B \to 0,
\]
where $\mathcal{A}_v^0$ is the connected component of $\mathcal{A}_v$ containing the origin, $T$ is a torus, and $B$ is an abelian variety. Failure of semistable reduction for $A$ essentially means that $T$ (which always exists as a linear algebraic group) has some factors isomorphic to $\G_a$.
\end{remark}
We note that one may attach a further numerical invariant to a place of semistable reduction:
\begin{definition}[Toric rank]
Let $A/K$ have semistable reduction at $v$. Write 
\[
0 \to T \to \mathcal{A}_v^0 \to B \to 0
\]
for the canonical exact sequence involving the identity component of the special fibre of $\mathcal{A}$ at $v$. The \emph{toric rank} of $A$ at $v$ is $\dim(T)$; it is equal to 0 precisely when $A$ has good reduction at $v$.
\end{definition}

Finally, we deal with the change of reduction type induced by an extension of the ground field. For the rest of this subsection we work over a local field; for varieties over number fields, the statements apply at a finite place $v$ after base change to $K_v$.

\begin{definition}[Potentially good, potentially semistable reduction]
Let $L/K$ be a finite extension of local fields. We say that $A$  \emph{acquires good reduction over $L$} (respectively \emph{acquires semistable reduction over $L$}) if $A \times_K L$ has good (respectively semistable) reduction.
We say that $A$ has \emph{potentially good} (respectively \emph{potentially semistable}) reduction if there exists some extension $L/K$ such that $A$ acquires good (semistable) reduction over $L$.
\end{definition}

\begin{remark}
When working with elliptic curves one often uses the term \emph{multiplicative reduction} to refer to what we call bad semistable reduction. In higher dimensions, however, the distinction between multiplicative and semistable is important: one says that $A$ has multiplicative reduction when it has semistable reduction of toric rank equal to the dimension of $A$.
We could also give an alternative definition of multiplicative reduction: an elliptic curve has multiplicative reduction if and only if the smooth locus in the reduction is isomorphic, possibly after a quadratic extension, to the multiplicative group $\G_m$ (compare with Definition \ref{def:semistable reduction}).
\end{remark}


Abelian varieties with semistable reduction are somewhat easier to understand than general abelian varieties, especially in terms of Galois representations. For this reason, the following important result due to Grothendieck \cite{MR354656} is often useful.

\begin{theorem}[Semistable reduction theorem, local case]
Let $A/K$ be an abelian variety over a local field. Then $A$ has potentially semistable reduction.
\end{theorem}
\begin{corollary}[Semistable reduction theorem, global case]
An abelian variety over a number field $K$ has potentially semistable reduction at all finite places of $K$.
\end{corollary}


We also mention that the property of having good or semistable reduction is preserved under extension of the ground field.
\begin{theorem}
Let $A/K$ be an abelian variety over a local field.
\begin{enumerate}[label={\arabic*.}, wide, labelindent=0pt]
\item Suppose that $A$ has good reduction. Then $A$ has semistable reduction.
\item Suppose that $A$ has good reduction. Then for any finite extension of fields $L/K$ the abelian variety $A_L$ also has good reduction.
\item Suppose that $A$ has bad semistable reduction. Then for any finite extension of fields $L/K$ the abelian variety $A_L$ has bad semistable reduction.
\end{enumerate}
\end{theorem}

We conclude this section by giving some useful information on the reduction type of a Jacobian. To do so, we need a notion of good and semistable reduction also for curves.
\begin{definition}
Let $C$ be a nice curve over a local field $K$. We say that $C$ has \emph{good reduction} if there exists a smooth proper relative curve over $\Spec(\Oring_{K})$ with generic fibre isomorphic to $C$. We say that $C$ has \emph{semistable reduction}, or that it \emph{admits a semistable model}, if there is a proper flat model $\mathcal{C}$ over $\Oring_{K}$ whose special fibre has only ordinary double points as singularities.
\end{definition}

\begin{theorem}[Good reduction of Jacobians]
Let $C/K$ be a nice curve of genus $g \geq 2$ over a local field. The Jacobian $J$ of $C$, considered as an abelian variety over $K$, has good reduction whenever $C$ does, but the converse implication does not hold: it is possible for $J$ to have good reduction even when $C$ has bad reduction. Moreover, if $C$ has semistable reduction, then so does $J$.
\end{theorem}

Finally, we present a result that determines the toric rank of $\Jac(C)$ in the case where $C$ has semistable reduction. To state it, we need the notion of \emph{dual graph}.

\begin{definition}[Dual graph of the special fibre]
Let $K$ be a local field with residue field $\F$ and let $C/K$ be a nice curve with semistable model $\mathcal{C} \to \Spec \OK$. The \emph{dual graph of the special fibre} of $\mathcal{C}$, denoted by $X(\mathcal{C} \times_{\OK} \overline{\F})$, is the multi-graph having a vertex for every irreducible component of $\mathcal{C} \times_{\OK} \overline{\F}$ and an edge (possibly, a loop) between components $C_i$ and $C_j$ if $C_i$ and $C_j$ meet at a double point. 
\end{definition}

\begin{remark}
In a picture of the special fibre, one draws the irreducible components as curves meeting at double points; in the dual graph, these components become vertices and the double points become edges.
\end{remark}

The next statement describes the toric rank of the reduction of $\operatorname{Jac}(C)$ in terms of the group $H^1(X(\mathcal{C}_{\overline{\mathbb{F}}}),\Z)$: this is the first singular cohomology group of the topological realisation of the graph $X(\mathcal{C}_{\overline{\mathbb{F}}})$, obtained by realising each edge as an interval and gluing its endpoints to the corresponding vertices. If $X(\mathcal{C}_{\overline{\mathbb{F}}})$ is connected and has $r$ vertices and $s$ edges, this group is free of rank $s-r+1$, the number of independent cycles.
\begin{theorem}[{\cite[Example 8 p.~246]{MR1045822}}]
Let $K$ be a local field with residue field $\F$ and let $C/K$ be a nice curve with semistable reduction. Let $\mathcal{C}/\OK$ be a semistable model of $C$. The Jacobian $\operatorname{Jac}(C)$ has semistable reduction, and its toric rank is equal to the rank of $H^1(X(\mathcal{C}_{\overline{\F}}), \Z)$.
\end{theorem}

\begin{example}[Hyperelliptic curves with toric rank 1]
Let $C \colon y^2=f(x)$ be a hyperelliptic curve over $\Q$ with $f\in\Z[x]$. Suppose that some prime $p \neq 2$ not dividing the leading coefficient of $f$ divides the discriminant of $f(x)$ exactly once: then $\Jac(C)$ has bad semistable reduction of toric rank $1$ at $p$. Indeed, the assumption implies that the equation $y^2=f(x)$ gives a semistable model of $C$ over $\Z_p$, with special fibre $y^2=(x-a)^2g(x)$, where $g(x)$ has no multiple roots modulo $p$ and $(g(x),x-a)=1$. This implies that the dual graph is a single vertex with a loop (the only irreducible component intersects itself), so that $H^1(X(\mathcal{C}_{\overline{\F}_p}), \Z)$ has rank 1.
\end{example}

\subsection{Torsion points and the Tate module}\label{subsec: torsion and Tate module}
We now turn to the study of torsion points of abelian varieties. We begin with the following definition:
\begin{definition}
Let $A$ be an abelian variety over the field $K$ and let $n$ be a positive integer. We define
$A[n]$ to be the kernel of $[n] \colon A(\overline{K}) \to A(\overline{K})$. We call $A[n]$ the \emph{group of $n$-torsion points of $A$}.
\end{definition}
\begin{remark}
As observed in Definition \ref{def:torsion AV 1}, $A[n]$ can in fact be defined as a group scheme over $K$: indeed, the map $[n]:A \to A$ is an isogeny (see Definition \ref{def:homs and isogenies}) defined over $K$, hence its kernel is a subgroup scheme of $A$ defined over $K$. This point of view leads to more natural (or at least more intrinsic) definitions, but we shall not pursue it further here. However, it is often fruitful from an arithmetic point of view to study the scheme $A[n]$ when $A$ is defined over a more general base scheme than a field.
\end{remark}

\begin{theorem}
Let $K$ be a field of characteristic zero and $A/K$ be a $g$-dimensional abelian variety. Then $A[n]$ is isomorphic to $\left(\Z/n\Z\right)^{2g}$ as an abstract group.
\end{theorem}
We give a quick argument; for a more conceptual one, see the proof of Theorem \ref{thm:StructureTorsionSubgroupsGeneral} below.
\begin{proof}
All the data that go into the definition of the abelian variety (the equations defining it in projective space, the multiplication map, the inverse, etc.) only involve finitely many elements of $K$. Hence, there is a subfield $F=\Q(c_1,\dotsc,c_r)\subseteq K$, finitely generated as a field extension of $\Q$, such that $A$ is defined over $F$. Any finitely generated field extension of $\Q$ can be embedded in $\C$, so it is enough to consider the case $K=\C$. The result then follows from Proposition \ref{prop:TorsionCpxAV}.
\end{proof}
\begin{remark}
The reduction to the case $K=\C$ in the above proof is a special case of the \emph{Lefschetz principle}: true statements of the first order theory of fields about $\C$ are true for any algebraically closed field $K$ of characteristic zero.
\end{remark}

More generally, one has the following: 
\begin{theorem}\label{thm:StructureTorsionSubgroupsGeneral}
Let $K$ be a field and $A/K$ be a $g$-dimensional abelian variety. Let $n$ be a positive integer prime to the characteristic of $K$: then $A[n]$ is isomorphic to $\left(\Z/n\Z\right)^{2g}$ as an abstract group.
\end{theorem}
\begin{proof}
For any positive divisor $n'$ of $n$, the isogeny $[n']$ is finite, étale (this uses the assumption $(n,\operatorname{char}(K))=1$, see Remark \ref{rmk:EtalenessIsogenies}), and of degree $(n')^{2g}$. Hence $\ker [n']$ consists of $(n')^{2g}$ distinct geometric points (it is an étale group scheme of rank $(n')^{2g}$, see Remark \ref{rmk: etale group schemes}). Since this holds for every divisor $n'$ of $n$, the only possible group structure for $A[n]$ is $(\Z/n\Z)^{2g}$.
\end{proof}

\begin{example}
The assumption $\operatorname{char}(K) \nmid n$ is essential. For example, if $E/\F_p$ is an elliptic curve, then $E[p](\overline{\F}_p)$ can never have order $p^2$: the group $E[p](\overline{\F}_p)$ can either be trivial, in which case we say that $E$ is \emph{supersingular}, or it can have order $p$, in which case we say that $E$ is \emph{ordinary}. We explain why $E[p](\overline{\F}_p)$ cannot consist of $p^2$ points. Fix a Weierstrass equation for $E$, say
\[
E \colon Y^2Z = X^3 + aXZ^2+ bZ^3 \subset \bbP^2_{\F_p}.
\]
For simplicity, we work with a short Weierstrass equation, hence implicitly assume $p \neq 2, 3$; the argument can be extended to all characteristics by considering a full Weierstrass equation. There is a Frobenius morphism 
\[
\begin{array}{cccc}
\operatorname{Frob} \colon & E & \to & E \\
& [X \colon Y \colon Z] & \mapsto & [X^p \colon Y^p  \colon Z^p],
\end{array}
\]
which can be seen to be purely inseparable and of degree $p$. Since $\operatorname{Frob}$ sends the origin $[0 {:} 1 {:} 0]$ to itself, it is a homomorphism of abelian varieties by Proposition \ref{prop:MorAV}. Since it is clearly a non-constant map between curves, it is surjective and therefore an isogeny. Proposition \ref{prop:FactorThroughDegree} then implies that we can write $[p] = V \circ \operatorname{Frob}$ for another isogeny $V$ (this is a well-studied isogeny called the \emph{Verschiebung} map). Multiplicativity of the degree shows that $p^2 = \deg [p] = \deg V \cdot \deg \operatorname{Frob} = \deg V \cdot p$, hence $\deg V =p$ and the inverse image of any point via $V$ consists of at most $p$ points. On the other hand, notice that $\operatorname{Frob}$ is injective on $E(\overline{\F}_p)$ (it acts as an automorphism on the coordinates), so $[p]^{-1}([0 {:} 1 {:} 0]) = V^{-1}([0 {:} 1 {:} 0])$ consists of at most $p$ points as claimed. Finally, since $[p]^{-1}([0 {:} 1 {:} 0])$ is a subgroup of $E(\overline{\F}_p)$ of $p$-power order, the bound on its size implies that it has order $1$ or $p$.
\end{example}

So far, we have discussed the torsion points of order $n$ for a fixed value of $n$. We now package together all torsion points of $\ell$-power order, for a fixed prime $\ell$, obtaining the so-called $\ell$-adic Tate module of the abelian variety.

\begin{definition}[Tate module]
Let $A$ be an abelian variety over a field $K$ and let $\ell$ be a prime number different from $\operatorname{char}(K)$. The \emph{$\ell$-adic Tate module of $A$} is the projective limit
\[
T_\ell(A) = \varprojlim_{n \to \infty} A[\ell^n],
\]
where the transition morphisms are given by multiplication by $\ell$.  The \emph{rational Tate module} is $V_\ell(A)  \coloneq  T_\ell(A) \otimes_{\Z_\ell} \Q_\ell$.
\end{definition}
\begin{remark}
By definition, an element of $T_\ell(A)$ is an infinite sequence $\{a_0, a_1, a_2, \dotsc\}$ of torsion points of $A$ such that
\begin{enumerate}[label={\arabic*.}, wide, labelindent=0pt]
\item $a_i \in A[\ell^i]$ for all $i \geq 0$;
\item $\ell a_{i+1} = a_i$ for all $i \geq 0$.
\end{enumerate}
\end{remark}

\begin{remark}\label{rmk:TellIsFree}
We have already seen that $A[\ell^n] \cong \left(\Z/\ell^n\Z\right)^{2g}$. This isomorphism is not canonical, but it is possible to make compatible choices of isomorphisms for different values of $n$. Passing to the limit in $n$ we obtain a (non-canonical) isomorphism $T_\ell(A) \cong \Z_\ell^{2g}$.
\end{remark}

It is sometimes useful to consider an adelic variant of the above construction, which for simplicity we only introduce in the case of fields of characteristic zero.
\begin{definition}\label{def:Adelic Tate module}
Let $K$ be a field of characteristic zero and let $A$ be an abelian variety over $K$. The \emph{adelic Tate module} is the projective limit of the groups $A[n]$ with transition maps given by $A[km] \xrightarrow{[k]} A[m]$. We have
\[
\widehat{T}(A) = \varprojlim_{n} A[n] = \prod_{\ell \text{ prime}} T_\ell(A) \cong \widehat{\Z}^{2g},
\]
where $\widehat{\Z}=\varprojlim_n \Z/n\Z \cong \prod_{\ell \text{ prime}}\Z_\ell$ is the \emph{profinite completion} of the ring of integers.
\end{definition}

\begin{remark}
Taking the limit over those $n$ that are powers of a fixed prime $\ell$, Theorem \ref{thm:PropertiesWeil} (1) implies the existence of a perfect bilinear Weil pairing
\[
\langle \cdot, \cdot \rangle_\ell \colon T_\ell(A) \times T_\ell(A^\vee) \to \varprojlim_k \mu_{\ell^k} =: \Z_\ell(1).
\]
The operation on $\Z_\ell(1)$ will be written additively. In particular, the trivial element will be denoted by $0$, so the non-degeneracy property of the Weil pairing becomes the statement that if $\langle x, y \rangle_\ell=0$ holds for all $y \in T_\ell(A^\vee)$, then $x=0$ (and symmetrically, $\langle x, y \rangle_\ell=0$ for all $x \in T_\ell(A)$ implies $y=0$).

By composing with an isogeny corresponding to a polarisation $\mathcal{L}$ we also obtain a skew-symmetric pairing
\[
\langle \cdot, \cdot \rangle_{\ell,\mathcal{L}} \colon T_\ell(A) \times T_\ell(A) \to \Z_\ell(1).
\]
Explicitly, we set
\[
\langle x,y\rangle_{\ell,\mathcal{L}}=\langle x,\lambda_{\mathcal{L}}(y)\rangle_\ell,
\]
where for $y=(y_0,y_1,y_2,\dotsc)$ we write
\[
\lambda_{\mathcal{L}}(y)=(\lambda_{\mathcal{L}}(y_0),\lambda_{\mathcal{L}}(y_1),\lambda_{\mathcal{L}}(y_2),\dotsc).
\]
Since $\lambda_{\mathcal{L}}$ is an isogeny, the induced map $V_\ell(A)\to V_\ell(A^\vee)$ is an isomorphism, so extending scalars to $\Q_\ell$ we obtain that $\langle x, y\rangle_{\ell,\mathcal{L}}$ gives a non-degenerate bilinear form on $V_\ell(A)$.

\end{remark}

\begin{corollary}\label{cor:RosatiInvolution}
The Rosati map from Definition \ref{def:Rosati-involution} satisfies $(\varphi)^{\dagger \dagger}=\varphi$.
\end{corollary}
\begin{proof}
Let $\varphi$ be an endomorphism of $A$ and fix a prime $\ell$ different from the characteristic of $K$.
Working in $V_\ell(A)$ we have
\[
\langle \varphi(P), Q \rangle_{\ell, \mathcal{L}}  = \langle \varphi(P), \lambda_{\mathcal{L}}(Q) \rangle_{\ell} = \langle P, \varphi^\vee \circ \lambda_{\mathcal{L}}(Q) \rangle_{\ell} = \langle P, \varphi^{\dagger}(Q) \rangle_{\ell,\mathcal{L}}.
\]
It follows that $\varphi \mapsto \varphi^\dagger$ is the adjunction map for the non-degenerate bilinear form $\langle \cdot, \cdot \rangle_{\ell, \mathcal{L}}$, hence it is involutive.
\end{proof}

\subsection{The Mordell--Weil theorem}
No introduction to the arithmetic theory of abelian varieties would be complete without a mention of the famous Mordell--Weil theorem:
\begin{theorem}[Mordell--Weil]
Let $A/K$ be an abelian variety over a number field. The group $A(K)$ is finitely generated, that is, there exist a number $r \in \N$ and a finite abelian group $T$ such that
\[
A(K) \cong \Z^r \oplus T.
\]
The number $r$ is called the \emph{(algebraic) rank} of $A/K$.
\end{theorem}
Generally speaking, the torsion subgroup $T$ is much better understood than the rank $r$, although important open questions remain also about the torsion subgroups of abelian varieties. In the rest of this chapter we will focus on torsion points, but the reader should be aware that the study of ranks of abelian varieties over number fields is a very active area of research.

\section{The Galois representations attached to an abelian variety}
In this section we define and study several Galois representations attached to abelian varieties, and more generally commutative connected algebraic groups, over fields of arithmetic interest -- we focus on number fields and finite fields, though $p$-adic fields also enter the picture. 
From now on, $A$ is a fixed commutative connected algebraic group (usually an abelian variety) over a field $K$ (usually either a number field or a finite field).

Fix a positive integer $N$. The crucial remark is that the coordinates of the points in the finite set $A[N]$ are algebraic and carry a natural action of the absolute Galois group $\abGal{K}$. It is clear that when a Galois automorphism acts on a point in $A(\overline{K})$ we get a new point in $A(\overline{K})$, and that the fixed points of this action are precisely the $K$-points of $A$ (in particular, $0_A$ is fixed under the Galois action). 
More precisely, a Galois automorphism also acts on the finite set $A[N]$. To show this, notice that -- since the equations that define $A$ and the group law have coefficients in $K$ -- one has
\[
\sigma(P+Q) = \sigma(P)+\sigma(Q)
\]
for all $\sigma \in \abGal{K}$ and $P, Q \in A(\overline{K})$. In particular, iterating the above formula, we get $\sigma([N]P)=[N] \sigma(P)$ for all positive integers $N$, which implies that if $P$ is an $N$-torsion point, then so is $\sigma(P)$. It follows that $\abGal{K}$ acts on $A[N]$ for every $N \geq 1$. Moreover, 
the Galois action is compatible with the obvious structure of $\Z/N\Z$-module of $A[N]$. Combining the previous remarks, we see that the following definition is well-posed:
\begin{definition}[Galois representation attached to $A$]
Let $A$ be a commutative connected algebraic group
over a field $K$ and let $N$ be an integer not divisible by the characteristic of $K$.
There is a natural action of $\abGal{K}$ on $A[N]$, or, which is the same, a natural representation
\[
\rho_{A, N} \colon \abGal{K} \to \Aut _{\Z/N\Z}\left(A[N] \right).
\]
\end{definition}

Since $A$ is fixed in what follows, we will usually write $\rho_N$ instead of $\rho_{A, N}$.
\begin{remark}
The representations $\rho_{N}$ are continuous: this amounts to saying that they factor through a finite quotient of $\abGal{K}$, and this is clear, because $\rho_{N}$ becomes trivial upon extending the base field to $K(A[N])$, the field obtained from $K$ by adjoining the coordinates of all the $N$-torsion points. Since there are only finitely many $N$-torsion points and each of them has algebraic coordinates, the field $K(A[N])$ is a finite extension of $K$, and $\rho_{A, N}$ factors via the finite group $\Gal(K(A[N])/K)$.
\end{remark}

\begin{remark}
Understanding the analogues of the representations $\rho_{N}$ can be quite complicated when the characteristic $p$ of $K$ divides $N$: we shall steer clear of these difficulties and focus on the case $p \nmid N$.
\end{remark}

Moreover, it is customary to study these Galois representations \emph{one prime at a time}: while not strictly necessary, this often simplifies matters. 

Fix a prime $\ell$. We would like to describe as precisely and as concretely as possible the image of the Galois representations $\rho_{\ell^n}$. In order to dispense with the dependence on $n$, it is often useful to pass to the limit $n \to \infty$. To do so formally, note that the representations $\rho_{\ell^n} \colon \abGal{K} \to \Aut  A[\ell^n]$ form a compatible system, in the sense that
\[
\ell \cdot \rho_{\ell^n}(\sigma)(P) = \rho_{\ell^n}(\sigma)(\ell \cdot P) = \rho_{\ell^{n-1}}(\sigma)(\ell \cdot P)\quad \forall P \in A[\ell^n], \forall \sigma \in \abGal{K}.
\]
This allows us to take the inverse limit over the transition maps $[\ell] \colon A[\ell^n]\to A[\ell^{n-1}]$ and work with the so-called \emph{$\ell$-adic representation}
\[
\rho_{\ell^\infty} \colon \abGal{K} \to \Aut  \varprojlim_{n} A[\ell^n] = \Aut(T_\ell A).
\]
The representation $\rho_{\ell^\infty}$ is again continuous and will be our main object of study in this section. We shall use the informal term \emph{image of Galois} to refer to either the groups
\[
G_{\ell^n}  \coloneq  \rho_{\ell^n}(\abGal{K}) \subseteq \Aut(A[\ell^n])
\]
or their $\ell$-adic counterpart,
\[
G_{\ell^\infty}  \coloneq  \rho_{\ell^\infty}(\abGal{K}) \subseteq \Aut(T_\ell(A)).
\]
These groups do of course depend on $\ell$ but -- in a sense that will be made more precise later (see Section \ref{sect: MT conjecture} for the case of abelian varieties) -- they are conjectured to be \emph{very similar} to each other.

\begin{notation}
Let $A/K$ be an abelian variety of dimension $g$ and let $\ell \neq \operatorname{char}(K)$ be a prime number.
We fix once and for all a $\Z_\ell$-basis of $T_\ell(A)$ (recall from Remark \ref{rmk:TellIsFree} that $T_\ell(A)$ is free of rank $2g$ over $\Z_\ell$). Upon reduction modulo $\ell^n$, this basis also induces a $\Z/\ell^n\Z$-basis of the free module $A[\ell^n] \cong (\Z/\ell^n\Z)^{2g}$.
With this choice, the groups $G_\ell$ and $G_{\ell^\infty}$ can respectively be identified with subgroups of $\GL_{2g}(\F_\ell)$ and $\GL_{2g}(\Z_\ell)$; in what follows we shall make this identification without further comment.
\end{notation}

\begin{remark}
For this remark, assume that $K$ is equipped with an embedding into $\C$. When $A$ is an arbitrary commutative connected algebraic group (and not necessarily an abelian variety), one can show more generally that $T_\ell(A)$ is a free $\Z_\ell$-module of rank equal to $b_1(A)$, the first Betti number of the complex manifold $A(\C)$. Recall that the first Betti number is the dimension of the $\Q$-vector space $H_1(A(\C), \Q)$. We can therefore consider the $\ell$-adic Galois representation of $A$ as taking values in $\operatorname{GL}_{b_1(A)}(\Z_\ell)$.
\end{remark}

\begin{example}[Cyclotomic character]\label{def:cyclotomic character}
    In the case $A=\G_m$, one has $A(\C)=\C^\times$, whose first homology group is $1$-dimensional. Since the $N$-torsion of $\G_m$ is the group $\mu_N$ of $N$-th roots of unity, the $\ell$-adic Tate module $T_\ell \G_m$ of $\G_m$ is simply the inverse system of all $\ell^n$-th roots of unity, which is indeed a $1$-dimensional $\Z_\ell$-module. The $\ell$-adic Galois representation amounts to a compatible collection of homomorphisms
    \[
    \chi_{\ell^n} \colon \abGal{K} \to \Aut  \mu_{\ell^n} \cong (\Z/\ell^n\Z)^\times
    \]
    that satisfy
    \[
    \sigma(\zeta_{\ell^n}) = \zeta_{\ell^n}^{\chi_{\ell^n}(\sigma)}
    \]
    for all roots of unity $\zeta_{\ell^n}$ of order $\ell^n$ and all $\sigma \in \abGal{K}$. Passing to the limit, we obtain a character $\chi_{\ell^\infty} \colon \abGal{K} \to \Z_\ell^\times$ that satisfies
    \[
    \sigma(\zeta_{\ell^n}) = \zeta_{\ell^n}^{\chi_{\ell^\infty}(\sigma)} {\qquad \text{for every $n\geq 1$}}.
    \]
    Note that exponentiation of $\zeta_{\ell^n}$ by an $\ell$-adic number $\sum_{i \geq 0} c_i \ell^i$ (each $c_i$ in $\{0,\dotsc,\ell-1\}$) makes sense by setting 
    \[
    \zeta_{\ell^n}^{\sum_{i \geq 0} c_i \ell^i} = \zeta_{\ell^n}^{\sum_{i=0}^{n-1} c_i \ell^i}\,.
    \]
    This representation is called the \emph{cyclotomic character} and is of great importance in number theory, even outside the theory of Galois representations. 
\end{example}

Finally, it is also useful to introduce a version of these notions that takes into account all primes at the same time -- the so-called \emph{adelic} Galois representation:
\begin{definition}
    Let $K$ be a field of characteristic $0$ and $A/K$ be a connected, commutative algebraic group. The \emph{adelic Galois representation} is the inverse limit of all representations $\rho_{A, N}$ over the system of positive integers, ordered by divisibility. It is therefore a map
    \[
    \widehat{\rho}_A  \colon \abGal{K} \to \Aut  \widehat{T}(A),
    \]
    where the adelic Tate module $\widehat{T}(A)$ is defined analogously to Definition \ref{def:Adelic Tate module}. By fixing a basis of $\widehat{T}(A) \cong \widehat{\Z}^{b_1(A)}$, we can consider the adelic Galois representation as a continuous homomorphism
    \[
    \widehat{\rho}_A  \colon \abGal{K} \to \operatorname{GL}_{b_1(A)}(\widehat{\Z}).
    \]
\end{definition}

The kernels of the continuous representations $\rho_{A, \ell^\infty}$ and $\widehat{\rho}_A$ are closed subgroups of $\abGal{K}$. We denote their fixed fields respectively by $K(A[\ell^\infty])$ and $K(A[\infty])$. The field $K(A[\ell^\infty])$ is the increasing union $\bigcup_{n \geq 1} K(A[\ell^n])$, and similarly $K(A[\infty]) = \bigcup_{N\geq 1} K(A[N])$. The field $K(A[N])$ is the smallest extension of $K$ over which all the points in $A[N]$ are defined; the fields $K(A[\ell^\infty])$ and $K(A[\infty])$ have a similar description. They are called the \emph{torsion fields} of $A$.

\subsection{Examples}
\begin{enumerate}
  \item For the multiplicative group $A=\G_m$, one has $A[n]=\mu_n$, the group of $n$-th roots of unity. The corresponding torsion fields are the cyclotomic extensions $K(\mu_n)$.
  \item If $K$ is a number field and $E/K$ is an elliptic curve without complex multiplication, Serre's open image theorem \cite{SerreOIT} states that the image of $\rho_{E,\ell^\infty}$ is open in $\GL_2(\Z_\ell)$ for every prime $\ell$. Consequently, the Galois groups $\Gal(K(E[\ell^n])/K)$ are ``as large as possible'' for almost all $\ell$. More precisely, Serre's result shows that the image of the adelic representation $\widehat{\rho}_E$ is open in $\operatorname{GL}_2(\widehat{\Z})$.
  \item For abelian varieties with complex multiplication, the torsion fields are much more structured and can be described using class field theory, see for example \cite[Chap.~II]{MR1312368} for the case of elliptic curves.
\end{enumerate}

\subsection{Algebraic cycles constrain the action of Galois}
A guiding principle in the study of Galois representations is that
the image of Galois is `small' only when there is a good reason for it to be so. 
We now try to explain what this principle means in practice. Recall that an \emph{algebraic cycle} is, roughly speaking, a formal linear combination of (irreducible reduced closed) subvarieties.

\subsubsection{Weil pairing: the image of Galois is contained in \texorpdfstring{$\operatorname{GSp}$}{GSp}}

Let $A$ be principally polarised (over $K$), so that we may identify $A$ with $A^\vee$. The $\ell$-adic Weil pairing
\[
\langle \cdot, \cdot \rangle_\ell \colon T_\ell A \times T_\ell A \to \Z_\ell(1)
\]
is Galois-equivariant, that is, for all $\sigma \in \abGal{K}$ and for all $t_1, t_2 \in T_\ell(A)$ we have
\[
\langle \sigma(t_1), \sigma(t_2) \rangle_\ell = \sigma( \langle t_1, t_2 \rangle_\ell ).
\]
This means in particular that $G_{\ell^\infty}$ (respectively,  $G_\ell$) lies in the subgroup of $\Aut(T_\ell(A))$ (respectively,~of $\Aut(A[\ell])$) consisting of those automorphisms $M$ that satisfy
\[
\langle Mv, Mw \rangle_\ell = \lambda(M) \langle v,w\rangle_\ell,
\]
for some $\lambda(M) \in \Z_\ell^\times$ (respectively, $\F_\ell^\times$). The automorphisms with this property form the general symplectic group:
\begin{definition}[General symplectic group]
Let $(V,\langle \cdot, \cdot \rangle)$ be a finite-rank free module over a (commutative unital) ring $R$, endowed with a non-degenerate alternating bilinear form $\langle \cdot, \cdot \rangle$. The \emph{general symplectic group} of $(V,\langle \cdot, \cdot \rangle)$ is the subgroup of $\GL(V)$ given by
\[
\operatorname{GSp}(V ,\langle \cdot, \cdot \rangle) =\bigl\{ M \in \GL(V) \bigm\vert \exists \lambda(M) \in R^\times \colon \langle Mv, Mw \rangle = \lambda(M) \langle v,w \rangle \; \forall v,w \in V \bigr\}.
\]
The map $M \mapsto \lambda(M)$ is a homomorphism from $\operatorname{GSp}(V,\langle \cdot, \cdot \rangle)$ to $R^\times$; its kernel is by definition the \emph{symplectic group} of $(V, \langle \cdot, \cdot \rangle)$.
By abuse of notation we will usually denote $\operatorname{GSp}(V, \langle \cdot, \cdot \rangle)$ by $\operatorname{GSp}(V)$.
\end{definition}

\begin{remark}
Given our implicit choice of basis of $T_\ell(A)$ we shall usually consider $G_{\ell^\infty}$ as being a subgroup of $\GSp_{2g}(\Z_\ell)$.
\end{remark}

\begin{example}[Elliptic curves]
As with other higher-dimensional phenomena, the fact that $G_{\ell^\infty}$ is contained in $\GSp_{2g}(\Z_\ell)$ rather than just $\GL_{2g}(\Z_\ell)$ is not easy to notice when $g=1$. Indeed, it is immediate to check that for $V=\Z_\ell^2$, equipped with the standard symplectic form $\Bigl\langle \begin{pmatrix}
x_1 \\ y_1
\end{pmatrix}, \begin{pmatrix}
x_2 \\ y_2
\end{pmatrix} \Bigr\rangle = x_1y_2-y_1x_2$, one has $\GSp(V)=\GL(V)$: in the case of elliptic curves the existence of the Weil pairing does not impose any restriction on the image of Galois.
\end{example}

\begin{remark}\label{rmk:AlgCycleWeilPairing}
This restriction on the image of Galois should be thought of as being a manifestation of the existence of a polarisation, which in turn is an algebraic cycle on $A \times A^\vee$.
Indeed, let $\Gamma \subset A \times A^\vee$ be the graph of the given principal polarisation $\lambda_{\mathcal{L}}$. Since $\lambda_{\mathcal{L}}$ is a homomorphism of abelian varieties, we have $\lambda_{\mathcal{L}}(A[\ell^n]) \subseteq A^\vee[\ell^n]$, and when we consider the action of $\sigma \in \abGal{K}$ on $P \in A[\ell^n]$ we obtain
\[
\sigma( P, \lambda_{\mathcal{L}}(P) ) \in \sigma(\Gamma) = \Gamma, 
\]
so that $(\sigma(P), \sigma(\lambda_{\mathcal{L}}(P)))$ is again a point of $\Gamma$. This imposes a nontrivial constraint on the action of $\sigma$, which forces the inclusion $G_{\ell^\infty} \subseteq \GSp(T_\ell(A))$.
\end{remark}

\begin{remark}\label{rmk:WeilPairingMultiplier}
We have the useful formula
\[
\lambda(\rho_{\ell^\infty}(\sigma)) = \chi_{\ell^\infty}(\sigma),
\]
where $\lambda$ is the multiplier $\GSp(T_\ell(A)) \to \Z_\ell^\times$ and $\chi_{\ell^\infty}$ is the cyclotomic character. This is true essentially by definition: indeed we know that
\[
\langle \rho_{\ell^{\infty}}(\sigma)(P), \rho_{\ell^{\infty}}(\sigma)(Q) \rangle_{\ell} = \sigma \left( \langle P,Q \rangle_{\ell} \right),
\]
and on the other hand the Galois action on the root of unity $\langle P,Q \rangle_{\ell}$ is by definition the cyclotomic character (see Definition \ref{def:cyclotomic character}). It is also useful to notice that for a $g$-dimensional abelian variety $A$ we have an equality between the group homomorphisms $\GSp_{2g}(V) \to R^\times$ given by $\lambda^g$ and by $\det$.
\end{remark}

\begin{remark}
It is interesting to understand what happens if $A$ carries no principal polarisation. We leave this as an exercise to the reader. As a hint, choose an ample line bundle $\mathcal{L}$ defined over $K$, and put $d=\deg\lambda_{\mathcal{L}}$. If $\ell\nmid d$, the polarisation $\lambda_{\mathcal{L}}$ induces an isomorphism between $T_\ell(A)$ and $T_\ell(A^\vee)$, so the induced pairing on $T_\ell(A)$ is perfect. At primes dividing $d$, the integral pairing need not be perfect (a simple example is given by a non-principal polarisation on an elliptic curve).
\end{remark}

\subsubsection{Endomorphisms}
Reasoning as in Remark \ref{rmk:AlgCycleWeilPairing}, we see that the existence of non-trivial $K$-endomorphisms of $A$ leads to restrictions on the image of Galois: indeed, if $\varphi \colon A \to A$ is an endomorphism of $A$ defined over $K$, for any torsion point $P \in A[\ell^n]$ and for every $\sigma \in \abGal{K}$ we have
\[
\sigma( \varphi(P) ) = \varphi(\sigma(P)),
\]
an equality which should be interpreted as saying that the knowledge of the action of $\sigma$ on $P$ is enough to determine the action of $\sigma$ on $\varphi(P)$. Clearly, in general this is a non-trivial restriction on the possible automorphisms of $T_\ell(A)$ which can lie in the image of Galois.

\begin{remark}
The above discussion applies in particular to the endomorphisms $[n]$ of $A$, but all that can be deduced from the existence of these endomorphisms is that $\sigma$ acts $\Z_\ell$-linearly on $T_\ell(A)$; this fact is already subsumed by our very notation, since we consider the image of Galois as a subgroup of $\GL(T_\ell(A))$.
\end{remark}

\begin{remark}
As in Remark \ref{rmk:AlgCycleWeilPairing}, the restrictions on the image of Galois coming from the existence of nontrivial endomorphisms can be interpreted as the existence of suitable algebraic cycles, namely the graph of the endomorphism itself in $A \times A$.
\end{remark}

A special case is that of \emph{decomposable} abelian varieties: if $A$ is isomorphic to $B \times C$, then $T_\ell(A)$ decomposes as $T_\ell(B) \oplus T_\ell(C)$, and the image of Galois will -- in a suitable basis -- act through block-diagonal matrices. If $A$ is only isogenous to a product $B \times C$ the same is true after tensoring with $\Q_\ell$: more precisely, while it is not true in general that $T_\ell(A) \cong T_\ell(B) \oplus T_\ell(C)$ (as Galois representations), we always have $V_\ell(A) \cong V_\ell(B) \oplus V_\ell(C)$. This means that the action of Galois may be put in block-diagonal form after a change of basis that may involve denominators. Note that the primes $\ell$ for which the isomorphism $T_\ell(A) \cong T_\ell(B) \oplus T_\ell(C)$ may fail are only those that divide the degree of every isogeny $A\rightarrow B \times C$.

\begin{remark}
Notice that if $A$ is isogenous to a product $B \times C$ then $A$ cannot be simple, and in particular it admits nontrivial endomorphisms. Indeed we have the following maps at our disposal:
\begin{enumerate}[label={\arabic*.}, wide, labelindent=0pt]
\item an isogeny $A \to B \times C$, by assumption;
\item isogenies $A \to A^\vee$, $A^\vee \to A$, $B \to B^\vee$, 
by Remark \ref{rmk:PolarisationsExist}.
\item an isogeny $B^{\vee} \times C^{\vee} \to A^\vee$, by duality
\end{enumerate}
We then obtain a homomorphism
\[
A \to B \times C \xrightarrow{\pi_1} B \to B^\vee \xrightarrow{\iota_1} B^\vee \times C^\vee \to A^\vee \to A
\]
whose image is easily seen to have dimension equal to $\dim(B) \neq 0, \dim A$. In particular, this endomorphism cannot be of the form $[n]$ for any integer $n$.
\end{remark}

\subsubsection{0-dimensional cycles: torsion points and isogenies}
Suppose that $A(K)_{\operatorname{tors}}$ is nontrivial and let $P \in A(K)_{\operatorname{tors}}$ be a point of order $\ell^n$. Then we have $\sigma(P) = P$ for all $\sigma \in \abGal{K}$, that is, the group $G_{\ell^n}$ is contained in the group of automorphisms of $A[\ell^n]$ that fix $P$.
More generally, if there is a nontrivial isogeny $f \colon A \to A'$ of degree $N$ defined over $K$, then $\ker f$ is a Galois-stable subgroup of $A[N]$. If $\ell^n \mid N$, this implies in particular that $G_{\ell^n}$ stabilises a nontrivial $\Z/\ell^n\Z$-submodule in $A[\ell^n]$, thus imposing another restriction on $G_{\ell^n}$.

We note, however, that this restriction is less geometric in nature than the previous ones we have considered, which arise from the existence of polarisations or endomorphisms, and are therefore largely determined by the geometry of $A_{\overline{K}}$. On the other hand, torsion points defined over $K$ are in some sense an `arithmetic accident', and of course, the number of torsion points defined over $K$ can increase if we replace $K$ by a finite extension. This may lead to the expectation that -- while the fine details of the image of Galois do of course depend on the field of definition $K$ -- the image \emph{up to finite index} (which by definition is insensitive to finite extensions of the ground field) should be determined purely by the geometry of $A$. This vague expectation is made precise in the Mumford--Tate conjecture, discussed in Section \ref{sect: MT conjecture}.

\subsection{Characteristic polynomials of Frobenius}
Characteristic polynomials of Frobenius are one of the main tools in the study of the Galois representations attached to abelian varieties, both from a theoretical point of view and for practical purposes. 
Before discussing them, however, we need to recall some basic terminology from algebraic number theory.

Let $v$ be a place of $K$. Fix a place $\overline{v}$ of $\overline{K}$ extending $v$ (or, equivalently, an embedding $\overline{K} \hookrightarrow \overline{K_v}$). The choice of $\overline{v}$ induces identifications of $\abGal{K_v}$ with the \emph{decomposition group} $D_v = D(\overline{v}/v)$, which contains the canonical inertia subgroup $I_v = I(\overline{v}/v)$. Different choices of $\overline{v}$ are conjugate under the Galois action, hence in particular, all the possible decomposition and inertia groups above $v$ are conjugate under Galois.

\begin{definition}
Let $A/K$ be an abelian variety and let $v$ be a place of $K$. Let $\ell$ be a prime number. We say that the representation $\rho_{\ell^n}$ (or $\rho_{\ell^\infty}$) is \emph{unramified at $v$} if for some (hence all) choices of an inertia subgroup $I_v = I(\overline{v}/v)$ we have $\rho_{\ell^n}(I_v)=\{1\}$ (respectively, $\rho_{\ell^\infty}(I_v)=\{1\}$).
\end{definition}


Luckily, $\rho_{\ell^\infty}$ is unramified at most places of $K$ (recall from Theorem \ref{thm:FinitenessBadReduction} that every abelian variety $A/K$ has good reduction at all but finitely many places of $K$):

\begin{theorem}\label{thm:UnramifiedReps}
For every $\ell$ and $n$, the representations $\rho_{\ell^n}$ and $\rho_{\ell^\infty}$ attached to $A$ are unramified away from $\ell$ and the places of bad reduction of $A$.
\end{theorem}
This result is most easily obtained as a consequence of the smooth proper base change theorem in étale cohomology, so we omit its proof. Among the many important consequences of Theorem \ref{thm:UnramifiedReps} is the well-posedness of the following crucial definition.

\begin{definition}
Let $v$ be a finite place of $K$ at which $A$ has good reduction. Then $D_v/I_v$ is pro-cyclic, 
generated by the \emph{(arithmetic) Frobenius at $v$}, which acts on the residue field by $x \mapsto x^{\card{\F_v}}$.
We denote by $\operatorname{Frob}_v$ any lift of this element to $\abGal{K}$: notice that $\operatorname{Frob}_v$ is well-defined only up to composition by elements of $I_v$ and the conjugation action of $\abGal{K}$.

Let now $\ell$ be a prime different from the residue characteristic of $v$. By Theorem \ref{thm:UnramifiedReps} we have that $\rho_{\ell^\infty}(I_v)=\{1\}$, which implies that $\rho_{\ell^\infty}(\operatorname{Frob}_v)$ is well-defined up to conjugation. Finally, since the characteristic polynomial of an endomorphism depends only on its conjugacy class, we may define
\[
f_{v,\ell}(t)  \coloneq  \operatorname{det} \left(t \operatorname{Id} - \rho_{\ell^\infty}(\operatorname{Frob}_v) \right),
\]
which is a well-defined polynomial in $\Z_\ell[t]$, independent of all the choices we made in defining $\operatorname{Frob}_v$.
\end{definition}

There is an obvious analogue of this definition in the case of finite fields:
\begin{definition}\label{def:FrobCharPolyFF}
Let $A/K$ be an abelian variety over a finite field of characteristic $p$. Choose a prime $\ell \neq p$: the \emph{characteristic polynomial of Frobenius} is defined as
\[
p_A(t) = \det(t\operatorname{Id} - \rho_{\ell^\infty}(\operatorname{Frob})) \in \Z_\ell[t].
\]
\end{definition}

The polynomial $f_{v,\ell}$ is attached to an abelian variety over a number field and a place $v$, whereas $p_A$ is attached to an abelian variety already defined over a finite field. If $\widetilde{A}_v/\F_v$ is the reduction of $A$ at a place of good reduction of characteristic different from $\ell$, then $f_{v,\ell}(t)=p_{\widetilde{A}_v}(t)$, by the compatibility of the Tate module with good reduction.

\subsubsection{Compatibility and the Weil conjectures}
The usefulness of the characteristic polynomials of Frobenius in the study of abelian varieties lies in the following deep result, which combines input by many people, but which was ultimately proved by Deligne \cite{MR340258}:
\begin{theorem}[Compatibility and the Weil conjectures]\label{thm:Weil}
The following hold:
\begin{enumerate}[label={\arabic*.}, wide, labelindent=0pt]
\item For each finite place $v$ of good reduction, the polynomial $f_{v,\ell}(t)$ has integral coefficients and is independent of the choice of the auxiliary prime $\ell\neq \operatorname{char}(\F_v)$.
\item The roots of $f_{v,\ell}(t) \in \Z[t] \subset \C[t]$ come in $g$ conjugate pairs $\tau_1, \overline{\tau_1}, \dotsc, \tau_g, \overline{\tau_g}$. Every $\tau_i$ is a \emph{Weil number}: its absolute value is $(\card{\F_v})^{1/2}$ under any embedding of $\overline{\Q}$ in $\C$.
\end{enumerate}
\end{theorem}

In the light of this theorem it makes sense to introduce the following definition:
\begin{definition}
With the above notation, we set
\[
f_v(t)  \coloneq  f_{v,\ell}(t)
\]
for any prime $\ell$ different from the residue characteristic of $v$.
\end{definition}

We check that this is compatible with what we discussed in Remark \ref{rmk:WeilPairingMultiplier}:
\begin{remark}
Let $\sigma  \coloneq  \rho_{\ell^\infty}(\operatorname{Frob}_v)$ for an arbitrary choice of $\operatorname{Frob}_v$.
The constant term of $f_v(t)$ is given by $\det(\sigma)$, which according to Remark \ref{rmk:WeilPairingMultiplier} is equal to $\lambda(\sigma)^g =  \chi_{\ell^\infty}(\operatorname{Frob}_v)^g= (\card{\F_v})^g$. On the other hand, by Theorem \ref{thm:Weil} we know that the constant term of $f_v(t)$ is equal to the product 
\[
\prod_{i=1}^g \tau_i \overline{\tau_i}=\prod_{i=1}^g |\tau_i|^2 = \prod_{i=1}^g \card{\F_v} = (\card{\F_v})^g.
\]
\end{remark}

\subsubsection{Characteristic polynomials of Frobenius for Jacobians}

We now explain how to determine $f_v(t)$ when $A$ is the Jacobian of a nice curve. The answer is again provided by work of Weil and Deligne; we summarise the relevant results in the following theorem.
\begin{theorem}[Zeta functions of curves over finite fields]
\label{thm:ZetaFunctionsFiniteFields}
Let $C/\F_q$ be a smooth projective curve over the finite field with $q=p^n$ elements. Define the local zeta function $Z(C,s)$ by the formula
\begin{equation}\label{eq:Z-function-curve}
    Z(C,s) = \exp\Bigl( \sum_{m=1}^\infty \frac{\card{C(\F_{q^m})}}{m} q^{-ms} \Bigr).
\end{equation}
Then the following hold:
\begin{enumerate}[label={\arabic*.}, wide, labelindent=0pt]
\item $Z(C,s)$ is a rational function of $t=q^{-s}$.
\item We can write
\[
Z(C,s) = \frac{P_C(t)}{(1-t)(1-qt)}
\]
for a certain polynomial $P_C(t)$ of degree $2g$ with integer coefficients.
\item The roots of $P_C(t) \in \Z[t] \subset \C[t]$ come in $g$ conjugate pairs $\tau_1, \overline{\tau_1}, \dotsc, \tau_g, \overline{\tau_g}$. Every $\tau_i$ is a Weil number of weight $-1$, that is, under every embedding in $\C$ it has absolute value $q^{-1/2}$, and $\tau_i \overline{\tau_i}=q^{-1}$.
\end{enumerate}
\end{theorem}

The usefulness of this theorem for the computation of $f_v(t)$ comes from the following additional result. In the statement, $f_v(t)$ denotes the characteristic polynomial of Frobenius attached to $\Jac(C)$ at $v$, while $P_{C_{\mathbb{F}_v}}(t)$ is defined by Theorem \ref{thm:ZetaFunctionsFiniteFields}.
\begin{theorem}
Let $C/K$ be a nice curve having good reduction at $v$. Then $f_v(t)$ agrees with $t^{2g}P_{C_{\F_v}}(1/t)$, where $C_{\F_v}$ is the curve over $\F_v$ deduced from (a smooth $\Oring_{K_v}$-model of) $C$ by reduction modulo $v$.
\end{theorem}

\begin{remark}\label{rmk:StructurePC}
By expanding formally the definition of $Z(C,s)$ and matching coefficients between this representation and the representation $Z(C,s) = \frac{P_C(t)}{(1-t)(1-qt)}$, one sees that the $g=g(C)$ numbers $\card{C(\F_q)}, \dotsc, \card{C(\F_{q^g})}$ are sufficient to uniquely determine the polynomial $P_C(t)$. Moreover, the symmetry condition (if $\tau$ is a root, then so is $\frac{1}{q\tau}$) ensures that the coefficients of $P_C(t)=\sum_{i=0}^{2g} a_i t^i$ satisfy $a_{g+i}=q^{i}a_{g-i}$.
\end{remark}

\begin{example}
We show how the method of Remark \ref{rmk:StructurePC} works in practice by expressing the polynomial $P_C(t)$ for a curve $C$ of genus $2$ over $\F_p$ in terms of $\card{C(\F_p)}$ and $\card{C(\F_{p^2})}$. We write $N_m$ for $\card{C(\F_{p^m})}$ and $P_C(t)=p^2t^4- pa_1t^3+ a_2 t^2 - a_1 t+1$: the symmetry of the polynomial follows from Remark \ref{rmk:StructurePC}, and the minus sign in front of the coefficient $a_1$ is for consistency with some familiar formulas in the case of elliptic curves, as we will see below. Expanding the ratio
$
\displaystyle \frac{P_C(t)}{(1-t)(1-pt)}
$
formally as a power series in $t$ we obtain
\[
Z(C,s) = 1+t (-a_1+p+1)+t^2 \big(-a_1 p-a_1+a_2+p^2+p+1\big)+O(t^3),
\]
while from the defining formula for $Z(C,s)$, namely Equation \eqref{eq:Z-function-curve}, we get
\[
\begin{aligned}
Z(C,s) & = \exp\Big( N_1 t + \frac{N_2}{2} t^2 + O(t^3) \Big) \\
& = 1 + N_1 t + \frac{1}{2}N_2 t^2 + \frac{1}{2} N_1^2 t^2 + O(t^3)
\end{aligned}
\]
from which we get 
\[
a_1 = p+1-N_1, \quad a_2=\frac{1}{2}(N_2 + N_1^2)-(p+1)N_1+p.
\]
While the formula for $a_2$ might not be very enlightening, the formula for $a_1$ should be familiar: it's precisely the definition of the so-called \emph{trace of Frobenius} (very often denoted by $a_p$) for elliptic curves.

As an example, consider the curve $C/\F_7 \colon y^2=x^5+x+2$. One may compute directly, for example using a computer algebra system, that $N_1=10$ and $N_2=46$. The polynomial $P_C(t)$ is therefore given by $49t^4+14t^3+2t+1$.
\end{example}

More generally, we have:
\begin{proposition}
Write $P_C(t)=\prod_{i=1}^{2g} (\alpha_i t-1)$ with $\alpha_i \in \C$. For every $m \geq 1$ we have
\[
\card{C(\F_{q^m})} = q^m+1 - \sum_{i=1}^{2g} \alpha_i^m.
\]
\end{proposition}

\begin{remark}[Number of $\F_q$-rational points of the Jacobian]
If we are interested in the arithmetic of the Jacobian itself, rather than the arithmetic of $C$, another formula which is often useful is the equality
\[\card{\Jac(C)(\F_q)}=P_C(1).\]
A quick proof of this equality goes through the Lefschetz trace formula. Write $J \coloneq  \operatorname{Jac}(C)$ and note first that $J(\F_q)$ is the subgroup of points of $J(\overline{\F_q})$ fixed by Frobenius. On the other hand, by the Lefschetz formula,
the number of fixed points of (geometric) Frobenius acting on $J(\overline{\F_q})$ is equal to the alternating sum of traces of $\operatorname{Frob}$ on $H^i_{\text{ét}}(J_{\overline{\F_q}},\Q_\ell)$, that is,
\[
\begin{aligned}
\card{J(\F_q)} & = \card{J(\overline{\F_q})^{\operatorname{Frob}}} = \sum_{i=0}^{2g} (-1)^i \operatorname{tr}\left( \operatorname{Frob} \bigm\vert H^i_{\text{ét}}(J_{\overline{\F_q}},\Q_\ell) \right)
\\ & = \sum_{i=0}^{2g} (-1)^i \operatorname{tr}\left( \operatorname{Frob} \bigm\vert \Lambda^i H^1_{\text{ét}}(J_{\overline{\F_q}},\Q_\ell) \right)
\\ & = \sum_{i=0}^{2g} (-1)^i \sum_{ \substack{H \subseteq \{1,\dotsc,{2g}\} \\ |H|=i}} \prod_{h \in H} \alpha_h = P_C(1).
\end{aligned}
\]
Note that, over $\C$, a $g$-dimensional abelian variety is topologically the real torus $(S^1)^{2g}$, whose cohomology algebra is the exterior algebra on its first cohomology group. The isomorphism $H^i_{\text{ét}}(J_{\overline{\F_q}},\Q_\ell)\cong\Lambda^i H^1_{\text{ét}}(J_{\overline{\F_q}},\Q_\ell)$ used above is the étale version of this fact; see \cite[Theorem 15.1]{MilneAV}.

\end{remark}

\section{Tate's conjecture on endomorphisms}
A powerful tool to relate an abelian variety $A$ to its Galois representations is given by Tate's conjecture on homomorphisms, now a theorem of Faltings \cite{FaltingsFiniteness} for number fields; Tate himself proved in \cite{MR206004} the version for finite fields.

\begin{theorem}[Faltings, Tate]\label{thm: Tate hom conjecture}
Let $A, B$ be abelian varieties defined over a common field $K$, where $K$ is either a number field or a finite field, and let $\ell \neq \operatorname{char}(K)$ be a prime. There is a natural isomorphism
\[
\Hom_K(A,B) \otimes_{\Z} \Q_\ell \cong \Hom_{\abGal{K}}(V_\ell(A), V_\ell(B)),
\]
where $\Hom_{\abGal{K}}(V_\ell(A), V_\ell(B))$ denotes the group of $\Q_\ell$-linear homomorphisms \[
g \colon V_\ell(A) \to V_\ell(B)
\]
that satisfy
\[
\sigma(g(v))=g(\sigma(v))\quad \forall v \in V_\ell(A), \forall \sigma \in \abGal{K}.
\]
\end{theorem}

\begin{remark}
Thanks to the work of many people over the years (with particularly important contributions by Zarhin), Tate's conjecture is now a theorem over any field finitely generated over its prime field. We note explicitly that Theorem \ref{thm: Tate hom conjecture} fails over $p$-adic fields. Indeed, for $\ell\neq p$, the Tate module of an abelian variety with good reduction over a $p$-adic field only sees its reduction. For example, if $E/\Q_p$ has good reduction and $\End_{\overline{\Q_p}}(E)=\Z$, the Galois action on $V_\ell(E)$ is topologically generated by a single element (the image of Frobenius), whose centraliser in $\End_{\Q_\ell}(V_\ell(E))$ has dimension at least $2$, whereas $\End_{\Q_p}(E)\otimes\Q_\ell$ has dimension $1$. The theorem also fails over $\R$ and $\C$, since the Galois groups of these fields are too small: for a non-CM elliptic curve, the Galois-equivariant endomorphism algebra has $\Q_\ell$-dimension $2$ when $K=\R$ and $4$ when $K=\C$, whereas $\End_K(E)\otimes\Q_\ell$ has dimension $1$.
\end{remark}

\begin{remark}
Theorem \ref{thm: Tate hom conjecture} is an incredibly useful and versatile result. In particular, it reduces the problem of studying maps between abelian varieties (which might seem highly non-linear objects) to that of understanding certain \emph{linear} objects. It is not unreasonable to consider this theorem as a sophisticated version of analytic uniformisation: over $\C$, an abelian variety $A$ can be represented as the quotient of the complex vector space $H^0(A, \Omega^1_A)^\vee$ by the lattice $H_1(A(\C),\Z)$, and maps between abelian varieties can be described as linear maps between $\C$-vector spaces that behave in a certain way with respect to the respective lattices. In the arithmetic context, the lattice $H_1(A(\C),\Z)$ is replaced by the $\Z_\ell$-lattice $T_\ell(A)$; compatibility with Galois action ensures that the maps in question are defined over the correct ground field.
\end{remark}

When $K$ is a finite field, $\abGal{K}$ is procyclic, generated by Frobenius, so the Galois action on $V_\ell(A), V_\ell(B)$ is completely determined by the characteristic polynomial of Frobenius (see Definition \ref{def:FrobCharPolyFF}). This leads to the following important result:
\begin{theorem}[Tate]
Let $A,B$ be abelian varieties over a finite field $K$. Let $p_A(t)$, $p_B(t)$ be the corresponding characteristic polynomials of Frobenius. Given any extension $L$ of $\Q$, write $p_A(t)=\prod_f f(t)^{a_f}$ and $p_B(t)=\prod_{f} f(t)^{b_f}$ for the factorisation of $p_A(t), p_B(t)$ in $L[t]$. Define
\[
r(p_A,p_B) = \sum_{f} \deg(f) a_f b_f.
\]
Then the following hold:
\begin{enumerate}[label={\arabic*.}, wide, labelindent=0pt]
\item $r(p_A,p_B)$ is independent of $L$;
\item $\operatorname{rank}_{\Z} \Hom_{K}(A,B) = r(p_A,p_B)$.
\end{enumerate}
\end{theorem}

\begin{remark}
One may take $L=\Q$ and factor $p_A, p_B$ into irreducible polynomials, or take $L$ to be a splitting field of $p_Ap_B$, in which case $r(p_A,p_B)=\sum_\alpha m_A(\alpha)m_B(\alpha)$, where $m_A(\alpha)$ and $m_B(\alpha)$ are the multiplicities of $\alpha$ as a root of $p_A$ and $p_B$, respectively. In particular, $\Hom_K(A,B)=0$ if and only if $p_A$ and $p_B$ have no common root in an algebraic closure.
\end{remark}




\section{The Mumford--Tate conjecture and independence of \texorpdfstring{$\ell$}{ell}}\label{sect: MT conjecture}
 Roughly speaking, the Mumford--Tate conjecture asserts that the images of the Galois representations attached to abelian varieties should be `large' -- more precisely, $\ell$-adically open in the $\Q_\ell$-points of a certain algebraic group -- and that this algebraic group should itself be completely determined by the Hodge classes on $A$ and its powers (the Hodge conjecture predicts that every rational Hodge class should arise from an algebraic cycle, and every algebraic cycle gives rise to a Hodge class. Since this is only an introductory discussion, we will phrase it in the more intuitive language of algebraic cycles). In its precise form, the conjecture is a little complicated to state, so we will limit ourselves to giving the following vague version:
 \begin{conjecture}[Mumford--Tate--Serre]
 Let $A$ be an abelian variety over a number field $K$. There is an algebraic subgroup $\operatorname{MT}(A)$ of $\GL_{2g,\Q}$, defined over $\Q$, with the following properties:
 \begin{enumerate}[label={\arabic*.}, wide, labelindent=0pt]
 \item $\operatorname{MT}(A)$ depends only on the \emph{geometry} of $A$, that is, it only depends on $A(\C)$ (independently of the choice of embedding $K \hookrightarrow \C$);
 \item there is a finite extension $K'$ of $K$ such that the following holds for every prime $\ell$: 
\begin{enumerate}[label={\alph*.}, wide, leftmargin=6pt, itemindent=6pt]
 \item viewing $\GL(T_\ell(A)) \cong \GL_{2g}(\Z_\ell)$ inside $\GL_{2g}(\Q_\ell)$, the group $\rho_{\ell^\infty}(\abGal{K'})$ is contained in $\operatorname{MT}(A)(\Q_\ell)$;
 \item $\rho_{\ell^\infty}(\abGal{K'})$ is open inside $\operatorname{MT}(A)(\Q_\ell)$ for the $\ell$-adic topology.
 \end{enumerate}
 \end{enumerate}
 \end{conjecture}

 Even though the full conjecture remains open, various parts of it and several special cases have been proven: there is a precise description of the group $\operatorname{MT}(A)$ (which does indeed depend only on $A(\C)$, and which by work of Deligne does not depend on the embedding of $K$ in $\C$); part (2a) has also been proven by Deligne, and for part (2b) one does at least know that $\rho_{\ell^\infty}(\abGal{K'})$ is open (in the $\ell$-adic topology) in the group of $\Q_\ell$-points of its Zariski closure. What we don't know is whether this Zariski closure is all of $\operatorname{MT}(A)_{\Q_\ell}$.
 The Mumford--Tate conjecture makes precise the following intuitions:
 \begin{enumerate}
 \item independence of $\ell$: indeed, it says that the different $G_{\ell^\infty}$ are all interpolated by a single algebraic group over $\Q$.
 \item the image of Galois is as large as it can be (up to finite index), once the geometric restrictions have been taken into account: loosely speaking, all the restrictions on $G_{\ell^\infty}$ should come from Hodge classes, and $\operatorname{MT}(A)$ is defined so as to capture the existence of all Hodge classes (hence, in particular, all algebraic cycles). Therefore part (2b) of the conjecture says that the image of Galois is -- up to some finite index -- as large as it can be, given the geometric restrictions encoded by $\operatorname{MT}(A)$.
 \end{enumerate}

\begin{example}
For the elliptic curve $E\colon y^2=x^3-x$, which has complex multiplication by $\Q(i)$, the Mumford--Tate group is the subgroup of $\GL_{2,\Q}$ given, in a suitable basis, by the matrices $\begin{pmatrix}a&-b\\b&a\end{pmatrix}$ with $a^2+b^2$ invertible. This is a two-dimensional torus, isomorphic to $\operatorname{Res}_{\Q(i)/\Q}\G_m$; over $\C$, it becomes isomorphic to $\G_m^2$. The fact that this group is smaller than the full $\GL_{2,\Q}$ reflects the fact that the Galois action must commute with the CM endomorphisms once these are defined over the ground field. Let now $A/K$ be an abelian surface such that $\End_{\overline{K}}(A)\otimes_{\Z}\Q=F$ is a real quadratic field. In this case, the Mumford-Tate group $\operatorname{MT}(A)$ is the centraliser of $F$ in $\GSp_{4,\Q}$. Over $\R$, the two embeddings $F \hookrightarrow \mathbb{R}$ give a decomposition of the $4$-dimensional first cohomology group of $A$ into two two-dimensional spaces, and $\operatorname{MT}(A)(\R)\cong\{(g_1,g_2)\in\GL_2(\R)^2:\det g_1=\det g_2\}$. The group $\operatorname{MT}(A)$ has dimension $7$, compared with dimension $11$ for $\GSp_{4, \Q}$.
\end{example}

\subsection{Kummer representations}

To conclude this introduction to the Galois representations attached to abelian varieties, we discuss more generally the case of Kummer and torsion-Kummer representations. We discuss only the most basic properties of these representations, and refer the reader to \cite{JonesRouseEndomorphisms, LombardoPeruccaReductions} for more information.

Fix a commutative connected algebraic group $A$ over a field $K$ and a rational point $P \in A(K)$. For every $N \geq 1$ we denote by $K_N$ the $N$-torsion field $K(A[N])$. We will similarly write $K_{\ell^\infty}$ for $K(A[\ell^\infty])$ and $K_\infty$ for $K(A[\infty])$. In what follows, we assume that $K$ is of characteristic $0$, or that we only work with $N$ coprime to its characteristic.

For every $N \geq 1$, we fix a point $Q_N \in A(\overline{K})$ such that $N Q_N = P$ (such a point will be called an $N$-division point of $P$). 
It is easy to see that the map
\[
\begin{array}{cccc}
\kappa_{P; N} \colon & \Gal(\overline{K}/K_N) & \to & A[N] \\
& \sigma & \mapsto & \sigma(Q_N) - Q_N
\end{array}
\]
is a well-defined homomorphism, independent of the choice of the point $Q_N$:
\begin{itemize}
    \item Well-defined: we have 
    \[
    N(\sigma(Q_N) - Q_N) = \sigma(NQ_N) - NQ_N = \sigma(P)-P=0\,,
    \]
    where the last equality uses that $P$ is rational, hence fixed by $\sigma$. This shows that $\kappa_{P; N}(\sigma)$ is indeed an $N$-torsion point.
    \item Homomorphism: we compute
    \[
    \begin{aligned}
    \kappa_{P; N}(\sigma_1 \sigma_2) &=\sigma_1\sigma_2(Q_N) - Q_N \\
    &=\sigma_1\left(\sigma_2(Q_N) - Q_N\right) + \sigma_1(Q_N) - Q_N \\
    &=\left(\sigma_2(Q_N) - Q_N\right) + \sigma_1(Q_N) - Q_N \\
    & = \kappa_{P; N}(\sigma_1) + \kappa_{P; N}(\sigma_2),
    \end{aligned}
    \]
    where the third equality uses that $\sigma_2(Q_N) - Q_N$ is an $N$-torsion point as shown above, hence fixed by $\Gal(\overline{K}/K_N)$ by definition of $K_N$.
    \item Independent of $Q_N$: if $Q'_N \in A(\overline{K})$ is another point such that $NQ'_N=P$, then $N(Q_N - Q'_N)=P-P=0$, hence $Q_N$ and $Q'_N$ differ by an $N$-torsion point $T$. Since by definition $T$ is fixed by $\Gal(\overline{K}/K_N)$, we have
    \[
    \begin{aligned}
    \sigma(Q'_N)-Q'_N & = \sigma(Q_N+T)-(Q_N+T) \\
    & = \sigma(Q_N) + \sigma(T)-(Q_N+T) \\
    & = \sigma(Q_N) + T -(Q_N+T) \\
    & = \sigma(Q_N)-Q_N.
    \end{aligned}
    \]
\end{itemize}
Just as with torsion representations, we can also pass to the limit in $N$. To see this, let $M, N$ be positive integers with $M \mid N$. As it is immediate to check, the point $\frac{N}{M} Q_N$ is an $M$-division point of $P$, hence it can be used to define the representation $\kappa_{P; M}$. It follows that
\[
\begin{aligned}
\frac{N}{M} \kappa_{P; N}(\sigma) & = \frac{N}{M} \left( \sigma(Q_N) - Q_N\right) \\
& = \sigma\Bigl( \frac{N}{M} Q_N \Bigr) - \frac{N}{M} Q_N  \\
& = \sigma\left( Q_M \right) - Q_M  \\
& = \kappa_{P; M}(\sigma),
\end{aligned}
\]
which is the compatibility condition necessary to pass to the inverse limit (over the system of positive integers ordered by divisibility) and obtain an adelic Kummer representation 
\[
\kappa_{P} \colon \Gal(\overline{K}/K_\infty) \to \widehat{T}(A),
\]
where $\widehat{T}(A)$ is the adelic Tate module of Definition \ref{def:Adelic Tate module}. In a similar way, one can also construct the $\ell^\infty$-Kummer representation
\[
\kappa_{P; \ell^\infty} \colon \Gal(\overline{K}/K_{\ell^\infty}) \to T_\ell A.
\]

Note that, for every $N \geq 1$, the image of $\kappa_{P; N}$ can be canonically identified with the Galois group of $K_N\bigl( \frac{1}{N}P \bigr)$ over $K_N$, where $K_N\bigl( \frac{1}{N}P \bigr)$ is the field generated over $K_N$ by the coordinates of all the $N$-division points of $P$. Similar comments apply to the images of $\kappa_{P; \ell^\infty}$ and $\kappa_{P}$.


We now extend the representation $\kappa_{P; N}$, which so far we have only defined on $\Gal(\overline{K}/K_N)$, to the full Galois group $\abGal{K}$. We will call the resulting map the \emph{torsion-Kummer representation}.
We describe the construction for finite values of $N$, and leave it to the reader to work out the necessary adjustments for the $\ell$-adic and adelic cases. Fix as above an $N$-division point $Q_N$ of $P$. The torsion-Kummer representation is given by
\[
\begin{array}{cccc}
\omega_{P; N} \colon & \abGal{K} & \to & A[N] \rtimes \Aut  A[N] \\
& \sigma & \mapsto & \left( \sigma Q_N- Q_N, \; \sigma|_{A[N]} \right).
\end{array}
\]
Here the target group is the semi-direct product $A[N] \rtimes \Aut  A[N]$, where an element $\varphi \in \Aut  A[N]$ acts naturally on $A[N]$.

It is not hard to check that $\omega_{P; N}$ is a group homomorphism, that its composition with the natural projection $A[N] \rtimes \Aut A[N] \to \Aut  A[N]$ is the $N$-torsion representation, and that the restriction of $\omega_{P; N}$ to $\Gal(\overline{K}/K(A[N]))$ is the previously constructed Kummer representation. Note -- unlike the representations we previously constructed, which only depended on $P$ -- the map $\omega_{P; N}$ depends on the choice of the point $Q_N$. However, the kernel of $\omega_{P; N}$ is independent of this choice, and its fixed field is the field $K(A[N], \frac{1}{N}P )$ generated by the coordinates of all the $N$-torsion points of $A$ and of all the $N$-division points of $P$. In particular, $\omega_{P; N}$ identifies $\Gal\bigl( K(A[N], \frac{1}{N}P ) / K \bigr)$ with a subgroup of $A[N] \rtimes \Aut  A[N]$. 

\begin{remark}
The constructions in this section extend to any finite collection of $K$-rational points $P_1, \ldots, P_r$. Choosing $N$-division points $Q_{i,N}$, one defines $\kappa_{P_1,\dotsc,P_r;N}(\sigma)=(\sigma(Q_{i,N})-Q_{i,N})_{i=1,\dotsc,r}$ and $\omega_{P_1,\dotsc,P_r;N}(\sigma)=((\sigma Q_{i,N}-Q_{i,N})_{i=1,\dotsc,r},\sigma|_{A[N]})$. The targets are $A[N]^r$ and $A[N]^r\rtimes\Aut A[N]$, respectively, with the diagonal action in the latter case. The $\ell$-adic and adelic constructions can be carried out in the same way.\end{remark}

\medskip

One of the most important results in Kummer theory is that the representations $\kappa_{P; N}$ have large image whenever $P$ satisfies some basic necessary conditions. We refer the reader to \cite{RibetKummer, HindryLangConj} for more general statements, and simply give an important special case.

\begin{theorem}\label{thm: Kummer theory}
    Let $K$ be a number field and $A/K$ be the product of an abelian variety and a torus. Let $P \in A(K)$ be a point such that the set $\{nP : n \in \Z\}$ is Zariski-dense in $A$. There is a constant $c=c(A/K, P)$ such that for all $N \geq 1$ the image of the Kummer representation $\kappa_{P; N} \colon \Gal(\overline{K}/K(A[N])) \to A[N]$ has index dividing $c$.
\end{theorem}

\begin{example}
    Note that, when $\dim A = 1$, the assumption of Zariski-density simply amounts to the fact that $P$ is not a torsion point of $A$. We apply the theorem in the basic case $A=\G_m$ and $P=2 \in \Q^\times = \G_m(\Q)$. Since $P$ is not a torsion point, its multiples (in this case, powers) are Zariski-dense in $\G_m$. The torsion subgroup is $\G_m[N] \cong \mu_N$, the group of $N$-th roots of unity, and an $N$-division point of $P$ is simply an $N$-th root of $2$ in $\overline{\Q}$. The Kummer map is
    \[
    \begin{array}{cccc}
         \kappa_{P; N} \colon & \Gal(\overline{\Q}/\Q(\zeta_N)) & \to & \mu_N \\
         & \sigma & \mapsto & \frac{\sigma (\sqrt[^N]{2}\, )}{\sqrt[^N]{2}},
    \end{array}
    \]
    where $\sqrt[^N]{2}$ is a fixed $N$-th root of $2$.
    The conclusion of the theorem is then that there is a constant $c$ with the following property: for every $N \geq 1$, the degree of the field $\Q(\sqrt[^N]{2}, \zeta_{N})$ over $\Q(\zeta_{N})$ is at least $\frac{1}{c} N$. This is not too hard to prove directly, but note for example that $\kappa_{P; 8}$ is not surjective: this is due to the fact that $\sqrt{2}$ is contained in $\Q(\zeta_8)$.
\end{example}

The constant $c$ in Theorem \ref{thm: Kummer theory} has only been made effective in very few cases: algebraic tori \cite{PST21, PeruccaSgobbaKummer} and elliptic curves \cite{LombardoTrontoKummer, LombardoTrontoUniformBound}.
Finally, we note that Kummer theory has been applied to great effect to different problems, see for example classical results of Hasse \cite{Has65} on the distribution of prime divisors of certain recurrence sequences, \cite{HindryLangConj} for applications to a version of Lang's conjecture in products of abelian varieties and tori, and \cite{PinkOrderReduction, MR2473894, Per15, LombardoPeruccaReductions} for the problem of the variation of the order of the reductions modulo primes of the base field of a point of infinite order.

\medskip

\textbf{Acknowledgements.} I am grateful to Michael Daas and Antonella Perucca for their careful reading of the first version of this text. This document grew out of lecture notes I wrote for the \textit{Luxembourg Summer School on Galois representations} (July 2018). Many thanks go to Enis Kaya, Andrea Maffei, and Pietro Mercuri for their extensive feedback. Thanks are also due to Bas Edixhoven for an illuminating discussion about Jacobians, and to Gabor Wiese, Antonella Perucca, Shaunak Deo, Ilker Inam and Alexander Rahm for the organisation of the summer school.

\bibliographystyle{plain}
\bibliography{biblio}

\textsc{Dipartimento di Matematica, Università di Pisa. Largo Bruno Pontecorvo 5, 56127 Pisa (Italy)}

\textit{Email address:} \texttt{davide.lombardo@unipi.it}

\end{document}